\documentclass[11pt]{amsart}

\usepackage{amssymb,amsmath,amscd}
\usepackage{graphicx}
\usepackage{bbm}
\usepackage{a4wide}

\theoremstyle{plain}
\newtheorem{theorem}{Theorem}
\newtheorem{prop}[theorem]{Proposition}
\newtheorem{lemma}[theorem]{Lemma}
\newtheorem{coro}[theorem]{Corollary}
\newtheorem{fact}[theorem]{Fact}

\theoremstyle{definition}
\newtheorem{definition}{Definition}
\newtheorem{remark}{Remark}

\newcommand{\ts}{\hspace{0.5pt}}
\newcommand{\nts}{\hspace{-0.5pt}}
\newcommand{\AAA}{\mathbb{A}}

\newcommand{\RR}{\mathbb{R}\ts}
\newcommand{\QQ}{\mathbb{Q}\ts}
\newcommand{\ZZ}{{\ts \mathbb{Z}}}
\newcommand{\PP}{{\ts \mathbb{P}}}
\newcommand{\KK}{\mathbb{K}}
\newcommand{\SSS}{\mathbb{S}}
\newcommand{\TT}{\mathbb{T}}
\newcommand{\MM}{\mathbb{M}}
\newcommand{\NN}{\mathbb{N}}
\newcommand{\XX}{\mathbb{X}}
\newcommand{\YY}{\mathbb{Y}}

\newcommand{\cA}{\mathcal{A}}
\newcommand{\cB}{\mathcal{B}}

\newcommand{\cM}{\mathcal{M}}
\newcommand{\cL}{\mathcal{L}}

\newcommand{\cP}{\mathcal{P}}

\newcommand{\vG}{\varGamma}
\newcommand{\vL}{\varLambda}

\newcommand{\ii}{\mathrm{i}}
\newcommand{\ee}{\mathrm{e}}

\newcommand{\dd}{\, \mathrm{d}}

\newcommand{\oplam}{\mbox{\Large $\curlywedge$}}
\newcommand{\smoplam}{\mbox{$\curlywedge$}}

\newcommand{\exend}{\hfill $\Diamond$}

\DeclareMathOperator{\dens}{dens}
\DeclareMathOperator{\den}{den}
\DeclareMathOperator{\vol}{vol}

\DeclareMathOperator{\card}{card}

\DeclareMathOperator{\supp}{supp}
\DeclareMathOperator{\bigplus}{\text{\Large $+$}}

\begin{document}

\title[On weak model sets]
{On weak model sets of extremal density}

\author{Michael Baake}

\author{Christian Huck}
\address{Fakult\"at f\"ur Mathematik, Universit\"at Bielefeld, \newline
\hspace*{\parindent}Postfach 100131, 33501 Bielefeld, Germany}
\email{$\{$mbaake,huck$\}$@math.uni-bielefeld.de}


\author{Nicolae Strungaru}
\address{Department of Mathematical Sciences, MacEwan University, \newline
\hspace*{\parindent}10700 \ts 104 Avenue, Edmonton, AB, Canada T5J 4S2, and \newline
\hspace*{\parindent}Department of Mathematics, Trent University, \newline
\hspace*{\parindent}1600 West Bank Drive, Peterborough, ON, Canada K9L 0G2 }
\email{strungarun@macewan.ca, nicolaestrungaru@trentu.ca}

\begin{abstract}
  The theory of regular model sets is highly developed, but does not
  cover examples such as the visible lattice points, the $k$th
  power-free integers, or related systems.  They belong to the class
  of weak model sets, where the window may have a boundary of positive
  measure, or even consists of boundary only. The latter phenomena are
  related to the topological entropy of the corresponding dynamical
  system and to various other unusual properties.  Under a rather
  natural extremality assumption on the density of the weak model set,
  we establish its pure point diffraction nature. We derive an
  explicit formula that can be seen as the generalisation of the case
  of regular model sets.  Furthermore, the corresponding natural patch
  frequency measure is shown to be ergodic. Since weak model sets of
  extremal density are generic for this measure, one obtains that the
  dynamical spectrum of the hull is pure point as well.
\end{abstract}

\maketitle

\section{Introduction}

The theory of regular model sets, which are also known as cut and
project sets with sufficiently nice windows, is well established; see
\cite{TAO} and references therein for general background.  One
cornerstone of this class is the pure pointedness of the diffraction
measure \cite{Hof,Martin,BM}. Equivalently, this means that the
dynamical spectrum of the uniquely ergodic hull defined by the model
set is pure point as well; compare \cite{LMS,BL,LS}.  The regularity
of the window is vital to the existing proofs such as that in
\cite{Martin}, and also enters the characterisation of regular model
sets via dynamical systems \cite{BLM}.

For quite some time, systems such as the visible lattice points or the
$k$th power-free integers have been known to be pure point diffractive
as well \cite{BMP}. These point sets can also be described as model
sets, but here the windows are no longer regular. In fact, for each of
these examples, the window consists of boundary only, which has
positive measure, and many other properties of regular model sets are
lost, too. In particular, there are many invariant probability
measures on the orbit closure (or hull) of the point set under the
translation action of the lattice.  Yet, as explicit recent progress
has shown, the natural cluster (or patch) frequency measure of this
hull is ergodic and the visible points are generic for this measure
\cite{BH}. Consequently, the dynamical spectrum is still pure point,
by an application of the general equivalence theorem \cite{BL}. Since
this example is one out of a large class with similar properties, it
is natural to ask for a general approach that includes all of
them. Such a class is provided by \emph{weak model sets}, where one
allows more general windows.  This name was coined by Moody
\cite{M-beyond,Moody}, see also \cite[Rem.~7.4]{TAO}, and apparently
was first looked at by Schreiber \cite{Schreiber}.

It is the purpose of this paper to derive some key results for weak
model sets.  To this end, we begin with the visible lattice points as
a motivating example. Then, we start from the general setting of model
sets for a general cut and project scheme $(G,H,\cL)$, see
Eq.~\eqref{eq:candp} below for a definition, but investigate the
diffraction properties for the larger class of windows indicated
above.  It turns out that this is indeed possible under one fairly
natural assumption, namely that of maximal or minimal density for a
given van Hove averaging sequence in the group $G$.  This assumption
guarantees pure point diffractivity (Theorems~\ref{thm:weak-pp} and
\ref{thm:min-pp}). 

In a second step, we analyse the ergodicity of the cluster frequency
measure for the very van Hove sequence, which then results in the
dynamical properties of the hull we are after.  In particular, we
establish that weak model sets of extremal density have pure point
dynamical spectrum, and calculate the latter. Finally, we apply our
results to coprime lattice families, which encompasses the $k$-visible
lattice points in $d$-space as well as other examples of arithmetic
origin, such as $k$-free or (coprime) $\mathcal{B}$-free integers
\cite{PH,BH,ELD,BKKL,KLW} and their generalisations to analogous
systems in number fields \cite{CV,BH,BK}.  This way, we demonstrate
that and how the theory of weak model sets provides a natural
framework for a unified treatment of such systems.

In parallel to our approach, Keller and Richard \cite{KR} have
developed an alternative view on model sets via a systematic
exploitation of the torus parametrisation for such systems; compare
\cite{BHP,HRB,Martin,BLM}. Their work includes weak model sets and
provides an independent way to derive several of our key results. In
this sense, the two approaches are complementary and, in conjunction,
give a more complete picture of a larger class of model sets than
understood previously, both concretely and structurally.

\section{Preliminaries and background}

Our general reference for background and notation is the recent
monograph \cite{TAO}. Here, we basically summarise some key concepts
and their extensions in the generality we need them.  Let $G$ be a
locally compact Abelian group (LCAG), and denote the space of
translation bounded (and generally complex) measures on $G$ by
$\cM^{\infty} (G)$. Here and below, measures are viewed as linear
functionals on the space $C_{\mathsf{c}} (G)$ of continuous functions
with compact support, which is justified by the general Riesz--Markov
theorem; see \cite{BF} for general background. In this setting, we use
$\mu (g)$ and $\int_{G} g \dd \mu$ for an integrable function $g$ as
well as $\mu (A) = \int_{G} 1^{}_{A} \dd \mu$ for a Borel set $A$
exchangeably.  For a measure $\mu$, we define its twisted version
$\widetilde{\mu}$ by $\widetilde{\mu} (g) = \overline{\mu
  (\widetilde{g}\ts )}$ for $g\in C_{\mathsf{c}} (G)$ as usual, where
$\widetilde{g} (x) := \overline{g (-x)}$.

If $\mu$ is a \emph{finite} measure on $G$, we define its norm as $ \|
\mu \| = \lvert \mu \rvert (G)$, where $\lvert \mu \rvert$ denotes the
total variation of $\mu$.  More generally, for $\nu \in \cM^{\infty}
(G)$ and any compact set $K\subseteq G$, we define
\[
     \| \nu \|^{}_{K} \, = \, \sup_{t\in G}\, \lvert \nu \rvert (t+K) \ts .
\]
It is clear that $\nu \in \cM^{\infty} (G)$ means $ \| \nu \|^{}_{K} <
\infty$ for any compact $K\subseteq G$.

\begin{fact}\label{fact:bounds}
  Let\/ $\mu$ be a finite measure on\/ $G$, let\/ $\nu \in
  \cM^{\infty} (G)$ and\/ $g\in C_{\mathsf{c}} (G)$.  If\/ $\supp (g)
  \subseteq K$, with\/ $K\subseteq G$ compact, one has the estimates
\[
    \| \mu * \nu * g \|^{}_{\infty} \, \leqslant \,
    \| \mu \| \, \| \nu \|^{}_{K} \, \| g \|^{}_{\infty} 
    \quad \text{and} \quad
    \lvert (\mu * \nu) (g) \rvert  \, \leqslant \,
    \| \mu \| \, \| \nu \|^{}_{K} \, \| g \|^{}_{\infty} \ts .
\]
\end{fact}

\begin{proof}
  Since $\nu*g$ defines a continuous function, $\| \nu * g
  \|^{}_{\infty} \leqslant \| \nu \|^{}_{K} \, \| g \|^{}_{\infty}$
  follows from standard arguments. Then, one finds
\[
\begin{split}
   \big\lvert \bigl( \mu * \nu * g \bigr) (x) \big\rvert
   \, & = \, \left| \int_{G} \bigl( \nu * g \bigr) (x-y) 
   \dd \mu (y) \right| \, \leqslant  \int_{G}
   \big\lvert \bigl( \nu * g \bigr) (x-y) \big\rvert
   \dd \lvert \mu \rvert (y) \\[1mm]
   & \leqslant  \int_{G} \| \nu * g \|^{}_{\infty} 
   \dd \lvert \mu \rvert (y) \, = \,
   \| \nu * g \|^{}_{\infty} \, \| \mu \|
   \, \leqslant \, \| \mu \| \, \| \nu \|^{}_{K} \,
   \| g \|^{}_{\infty} \ts ,
\end{split}
\]
which proves the first claim.

Next, observe that $(\mu*\nu)(g) = \int_{G\times G}\, g (x+y) \dd \mu
(x) \dd \nu (y) = (\mu * \nu * g^{}_{\text{-}}) (0)$, where
$g^{}_{\text{-}}$ is defined by $g^{}_{\text{-}} (x) := g(-x)$, so
that $ \lvert (\mu * \nu) (g) \rvert \leqslant \| \mu * \nu *
g^{}_{\text{-}}\|^{}_{\infty}$, and the second claim follows from the
first because $\| g^{}_{\text{-}}\|^{}_{\infty} = \|g\|^{}_{\infty}$.
\end{proof}

Let $H$ be a compactly generated LCAG, hence (up to isomorphism) of
the form $\RR^{d} \times \ZZ^{n} \times \KK$ for some integers
$d,n\geqslant 0$ and some compact Abelian group $\KK$; compare
\cite[Thm.~9.8]{HR-Book}. We assume $H$ to be equipped with its Haar
measure $\theta = \theta^{}_{\! H}$, where we follow the standard
convention that this is Lebesgue measure on $\RR^{d}$, counting
measure on $\ZZ^{n}$ and normalised on compact groups, so
$\theta^{}_{\KK} (\KK) = 1$. The Haar measure on $G$ is denoted by
$\theta^{}_{\nts G}$, where we will use $\dd t$ instead of $\dd
\theta^{}_{\nts G} (t)$ for integration over (subsets of) $G$. Also,
we will write $\vol (A)$ instead of $\theta^{}_{\nts G} (A)$ for
measurable sets $A \subset G$.

The \emph{covariogram function} $c^{}_{W}$ of a relatively compact
Borel set $W\subseteq H$ is the real-valued function $c^{}_{W}$
defined by
\[
     c^{}_{W} (x) \, = \, \bigl( 1^{}_{W} * \widetilde{1^{}_{W}}
             \bigr) (x) \ts ,
\] 
where convolution is defined via $ \theta^{}_{\! H}$ as usual.
Note that the value at $0$ is given by
\begin{equation}\label{eq:cov-at-0}
     c^{}_{W} (0) \, = \int_{\nts\nts H}\, \lvert 1^{}_{W} (x) \rvert^{2}
     \dd \theta^{}_{\! H} (x) \, = \, \theta^{}_{\! H} (W) \ts .
\end{equation}

\begin{fact}\label{fact:cont}
  Let\/ $W$ be a relatively compact Borel set in a compactly generated
  LCAG\/ $H$. Then, the corresponding covariogram function\/
  $c^{}_{W}$ is bounded and uniformly continuous on\/ $H$.
\end{fact}

\begin{proof}
  Both $1^{}_{W}$ and $\widetilde{1^{}_{W}}$ are elements of $L^{1}
  (H) \cap L^{\infty} (H)$, whence $1^{}_{W} * \widetilde{1^{}_{W}}$
  is well-defined.  The convolution of an $L^{1}$ function with an
  $L^{\infty}$ function is uniformly continuous and bounded by
  standard arguments \cite[Thm.~1.1.6]{Rudin}.
\end{proof}

Next, we need a \emph{cut and project scheme} (CPS) as introduced in
\cite{Meyer}, coded by a triple $(G,H,\cL)$; see also
\cite{M-Nato,M-beyond,TAO} for background. Here, we use a LCAG $G$ as
\emph{direct space}, another LCAG $H$ as \emph{internal space}, and a
lattice $\cL \subset G\times H$ subject to some further restrictions
as follows,
\begin{equation}\label{eq:candp}
\renewcommand{\arraystretch}{1.2}\begin{array}{r@{}ccccc@{}l}
   & G & \xleftarrow{\,\;\;\pi\;\;\,} & G \times H & 
        \xrightarrow{\;\pi^{}_{\mathrm{int}\;}\,} & H & \\
   & \cup & & \cup & & \cup & \hspace*{-1ex} 
   \raisebox{1pt}{\text{\footnotesize dense}} \\
   & \pi(\cL) & \xleftarrow{\, 1-1 \,} & \cL & 
   \xrightarrow{\; \hphantom{1-1} \;} & \pi^{}_{\mathrm{int}}(\cL) & \\
   & \| & & & & \| & \\
   & L & \multicolumn{3}{c}{\nts\xrightarrow{\qquad\qquad\;\star
       \;\qquad\qquad}} 
       &  {L_{}}^{\star\nts} & \\
\end{array}\renewcommand{\arraystretch}{1}
\end{equation}
Here, $\pi$ and $\pi^{}_{\mathrm{int}}$ denote the natural
projections.  Since the lattice is located within $G\times H$ such
that its projection into $G$ is $1-1$, one inherits a well-defined
$\star$-map from $L$ into $H$, which will become important later on.
Note that, when $G$ is torsion-free, the $\star$-map has a unique
extension to $\QQ L$, which is a particularly useful property for
$G=\RR^{d}$. In our exposition below, we will further assume that $G$
is $\sigma$-compact and $H$ is compactly generated.

A \emph{projection set} (or cut and project set) in the strict sense
is any set of the form
\[
     \oplam (W) \, = \, \{ x \in L \mid x^{\star} \in W \}
\]
with $W\subseteq H$. Such a set is called a \emph{model set}, if $W
\subset H$ is relatively compact with non-empty interior. When
$\varnothing \ne W = \overline{W^{\circ}}$ is compact, the window is
called \emph{proper}. When, in addition, $\theta^{}_{\! H} (\partial
W) =0$, the model set is called \emph{regular}. In this situation, a
highly developed theory is at hand; see \cite{TAO} and references
therein for background. The known results easily generalise to
relatively compact windows, when $\theta^{}_{\! H} (W^{\circ}) =
\theta^{}_{\! H} (\, \overline{\! W \!} \,)$, which is often needed
for practical examples. Here, we are interested in the significantly
more general situation where one only demands $W \subseteq H$ to be a
relatively compact set with $\theta^{}_{\! H} (\, \overline{\!  W
  \!}\,) > 0$, without further conditions. The corresponding cut and
project set $\oplam (W)$ is then called a \emph{weak model set}, and
one generally has the chain of inclusions
\[
    \{ \text{regular model sets} \} \
    \, \subsetneq \, \{ \text{model sets}\}
   \, \subsetneq\, \{ \text{weak model sets}\} 
   \, \subsetneq \, \{ \text{projection sets} \}\ts .
\]
Note that weak model sets can have rather curious properties. In
particular, they need neither be Meyer sets nor even Delone sets. A
classic example, which we will discuss below again in some detail, is
provided by the visible points of a lattice in Euclidean space. Let us
also stress that our condition $\theta^{}_{\! H} (\, \overline{\! W
  \!}\,) > 0$ essentially excludes point sets with vanishing upper
density. We will not consider more general situations in this paper.

\begin{remark}
  If $W \subset H$ is relatively compact, there is a compact
  neighbourhood $K$ of $0\in H$ such that $W\subseteq K \subseteq H$.
  If we had started with a general LCAG $H$ in our CPS, we could now
  reduce the CPS to one with the group $H_{0}$ generated by $K$
  instead of $H$.  In this sense, our assumption that $H$ be compactly
  generated is no restriction.  \exend
\end{remark}

For our extensions below, we also need the concept of a
\emph{weighted} model set. By this we mean a marked set of the form
$\{ (x, h_{x}) \mid x \in \oplam (W) \}$ where the $h_{x}$ are real or
complex numbers, usually assumed bounded. Of particular relevance is
the case that the weights satisfy $h_{x} = c (x^{\star})$ with a
continuous, real- or complex-valued function $c$ on $H$. Particularly
nice properties emerge when $c$ is compactly supported
\cite{LR}. Moreover, if $c$ is also positive definite (or a linear
combination of functions of that class), one obtains a powerful
extension of the Poisson summation formula to weighted Dirac combs
\cite{TAO,RS}.

To formulate it, we need a dual to the CPS of Eq.~\eqref{eq:candp}.
First, given an LCAG $G$, its \emph{dual}, $\widehat{G}$,
 is the set of continuous
characters $\chi\! : \, G \xrightarrow{\quad} \SSS^{1}$, which is an
LCAG again, with multiplication of characters as group operation.  For
our purposes, it is advantageous to write the group additively, by
identifying a character $\chi (.) = \chi^{}_{u} (.)$ with a pairing
$\langle u , .\rangle$, so that $\chi^{}_{u} \chi^{}_{v} =
\chi^{}_{u+v}$ in analogy to $\chi^{}_{u} (x) = \ee^{2 \pi \ii ux}$ in
the important case $G=\RR^{d}$, where $ux$ is the standard inner
product in $\RR^{d}$. Now, using this additive notation, and observing
the natural isomorphism $\widehat{G\times H} \simeq \widehat{G}\times
\widehat{H}$, the \emph{dual} CPS \cite{M-Nato,M-beyond} is given by
\begin{equation}\label{eq:dual-candp}
\renewcommand{\arraystretch}{1.2}\begin{array}{r@{}ccccc@{}l}
   & \widehat{G} & \xleftarrow{\,\;\;\pi\;\;\,} & 
   \widehat{G} \times \widehat{H} & 
        \xrightarrow{\;\pi^{}_{\mathrm{int}\;}\,} & \widehat{H} & \\
   & \cup & & \cup & & \cup & \hspace*{-1ex} 
   \raisebox{1pt}{\text{\footnotesize dense}} \\
   & \pi(\cL^{\ts 0}) & \xleftarrow{\, 1-1 \,} & \cL^{\ts 0} & 
   \xrightarrow{\; \hphantom{1-1} \;} & 
   \pi^{}_{\mathrm{int}}(\cL^{\ts 0}) & \\
   & \| & & & & \| & \\
   & L^{0} & \multicolumn{3}{c}{\nts\xrightarrow{\qquad\qquad\;\star
       \;\qquad\qquad}} 
       &  {L_{}^{0}}^{\star\nts} & \\
\end{array}\renewcommand{\arraystretch}{1}
\end{equation}
without further restrictions on $\widehat{G}$ and $\widehat{H}$.
Here, to define $\cL^{\ts 0}$, we make use of the fact that $\cL$ from
the original CPS \eqref{eq:candp} has the form $\cL = \{ (x,x^{\star})
\mid x \in L \}$, which permits us to define
\[
    \cL^{\ts 0} \, := \{ (u,v) \in \widehat{G} \times \widehat{H} \mid
      \langle u, x \rangle \langle v, x^{\star}\rangle = 1
      \text{ for all } x \in L \} \ts ,
\]
which is a lattice for the new CPS; compare \cite[Sec.~5]{M-Nato} and
references therein for more. 

The important properties indicated in Eq.~\eqref{eq:dual-candp} are
inherited from the original CPS \cite{M-Nato}, so we have once again a
well-defined $\star$-map, for which we use the same symbol. In
particular, $\cL^{\ts 0}$ can also be written as $\cL^{\ts 0} = \{ (u,
u^{\star}) \mid u \in L^{0} \}$.  Under the isomorphism
$\widehat{G\times H} \simeq \widehat{G}\times \widehat{H}$, one sees
that $\cL^{\ts 0}$ becomes the annihilator of $\cL \subset G\times H$
in the dual group $\widehat{G\times H}$, and also that one has a
natural isomorphism
\[
    \widehat{\cL^{\ts 0}} \, \simeq \, 
    \TT = (G \times H)/\cL \ts .
\]
Note that, in the Euclidean setting, $\cL^{\ts 0}$ coincides with the
standard dual lattice $\cL^{*}$ of $\cL$.

In this setting, we have the following important result.

\begin{theorem}\label{thm:diffraction-formula}
  Consider a CPS\/ $(G,H,\cL)$ according to Eq.~\eqref{eq:candp}, and
  fix some\/ $c\in C_{\mathsf{c}} (H)$ that is a positive definite
  function on\/ $H$. Then, the weighted Dirac comb
\[
    \omega^{}_{c} \, := \sum_{x\in L} c (x^{\star}) \, \delta_{x}
\]
  is a translation bounded pure point measure that is Fourier
  transformable, with 
\[
    \widehat{\omega^{}_{c}} \, = \, \dens (\cL) 
    \sum_{u\in L^{0}} \widehat{c}\ts 
     (-u^{\star}) \, \delta_{u} \ts ,
\]
where\/ $L^{0} = \pi (\cL^{\ts 0})$ according to the dual CPS of
Eq.~\eqref{eq:dual-candp}. Here, $\widehat{\omega^{}_{c}}$ is a
translation bounded and positive pure point measure on\/
$\widehat{G}$.
\end{theorem}

\begin{proof}[Sketch of proof]
  The result is a consequence of the Poisson summation formula (PSF)
  together with the uniform distribution of the lifted points in the
  window; see \cite{Moody} and references therein. In fact, it is an
  interesting observation that the validity of the PSF, via Weyl sums,
  can be used to derive the uniform distribution --- the PSF thus
  appears in a double role \cite{RS}.  The factor $\dens (\cL)$ stems
  from the PSF, compare \cite[Thm.~9.1 and Lemma~9.3]{TAO} for the
  Euclidean case.
   
  The general version of the claim as stated here is proved in
  \cite{RS} as well as in \cite[Prop.~12.2]{Nicu-Band2}. Note that
  part (iii) of this proposition, which is what we need here, does
  \emph{not} use or need the assumption of Fourier transformability of
  $\omega^{}_{c}$.
\end{proof}

Given a $\sigma$-compact LCAG $G$, an \emph{averaging sequence} $\cA =
(A_{n})^{}_{n\in\NN}$ consists of relatively compact open sets $A_{n}$
with $\overline{A_{n}} \subset A_{n+1}$ for all $n\in\NN$ and
$\bigcup_{n\in\NN} A_{n} = G$. Here, $\sigma$-compactness of $G$ is
equivalent to the existence of such an averaging sequence. Now, $\cA$
is called \emph{van Hove} if, for any compact $K\subset G$,
\[
     \lim_{n\to\infty} \frac{\vol (\partial^{K}\! A_{n})}
     {\vol (A_{n})} \, = \, 0 \ts ,
\]
where, for an arbitrary open set $B\subset G$, $\partial^{K}\! B :=
\bigl(\ts\ts \overline{B+K}\setminus B\bigr) \cup
\bigl((B^{\mathsf{c}} - K)\cap \overline{B}\, \bigr)$, with
$B^{\mathsf{c}}$ the complement of $B$ in $G$, is the (closed)
$K$-boundary of $B$.  The existence of van Hove sequences in
$\sigma$-compact LCAGs is shown in \cite{Martin}. Note that each van
Hove sequence is also F{\o}lner, but not vice versa; compare the
discussion in \cite{BL} and references therein.

Averaging sequences are needed to define the density of a point set 
$\vL\subset G$,
\[
    \dens (\vL) \, := \lim_{n\to\infty}
    \frac{\card (\vL \cap A_{n}) }{\vol (A_{n})} \ts ,
\]
provided the limit exists. More generally, if the existence of $\dens
(\vL)$ is not clear, one has to work with lower and upper densities
according to
\begin{equation}\label{eq:lower-upper-dens}
   \underline{\dens} (\vL) \, := \, \liminf_{n\to\infty}
   \frac{\card (\vL \cap A_{n})}{\vol (A_{n})}
   \quad \text{and} \quad
   \overline{\dens} (\vL) \, := \, \limsup_{n\to\infty}
   \frac{\card (\vL \cap A_{n})}{\vol (A_{n})} \ts ,
\end{equation}
which always exist, with $0\leqslant\underline{\dens} (\vL)
\leqslant\overline{\dens} (\vL)\leqslant\infty$. When $\vL$ is
uniformly discrete, one has $\overline{\dens} (\vL)<\infty$.
Moreover, as a result of \cite{HR,Nicu-Band2}, one also gets the
following estimate.
\begin{fact}\label{fact:dens-bound}
  Let\/ $\vL=\oplam (W)$ be a projection set for the CPS of
  Eq.~\eqref{eq:candp}, with relatively compact window\/ $W\subset
  H$. Then, one has
\[
    \dens (\cL)\, \theta^{}_{\! H} (W^{\circ}) \, \leqslant \,
    \underline{\dens} (\vL) \, \leqslant \,
    \overline{\dens} (\vL) \, \leqslant \,
    \dens (\cL)\, \theta^{}_{\! H} (\,\overline{\! W\!}\,) \ts , 
\]
relative to any fixed van Hove sequence\/ $\cA$ in\/ $G$.  When\/ $
\underline{\dens} (\vL) = \overline{\dens} (\vL)$, the density of\/
$\vL$ exists, relative to\/ $\cA$, and satisfies the corresponding
inequality. \qed
\end{fact}

In our context, we also need the van Hove property for a meaningful
definition of the autocorrelation of a translation bounded measure
$\omega\in\cM^{\infty} (G)$ with respect to $\cA$. Consider
\[
     \gamma^{(n)}_{\omega} \, := \,
     \frac{\ts \omega|^{}_{A_{n}} \! * \widetilde{\omega|^{}_{A_{n}}}}
     {\vol (A_{n})} \ts ,
\]
which gives a well-defined sequence of positive definite measures on
$G$. As a consequence of the translation boundedness of $\omega$, this
sequence has at least one vague accumulation point, each of which is
called an autocorrelation measure of $\omega$; see \cite{Hof} or
\cite{TAO} for background. If only one such accumulation point exists,
$\gamma^{}_{\omega} := \lim_{n\to\infty} \, \gamma^{(n)}_{\omega}$
exists and is called the \emph{autocorrelation} of $\omega$ relative
to $\cA$.

By construction, any autocorrelation measure $\gamma$ is a positive
definite measure, and hence Fourier transformable by standard
arguments \cite{BF}. Its Fourier transform, $\widehat{\gamma}$, is
then a translation bounded positive measure on the dual group,
$\widehat{G}$, and called the (corresponding) \emph{diffraction
  measure}. If the autocorrelation for $\cA$ is unique, then so is the
diffraction measure, which is thus also referenced via $\cA$. It is
this measure $\widehat{\gamma}$ that we will explore below, and
ultimately use to gain access to the dynamical spectrum of a natural
dynamical system defined via $\omega$.

\section{Visible lattice points as guiding example}\label{sec:visible}

By definition, the visible points of a lattice in Euclidean space are
the lattice points that are visible from the origin. Although all
results below hold in much greater generality, we prefer to begin with
the visible points
\[
    V\,=\,\{(x,y)\in\ZZ^{2} \mid \gcd(x,y)=1\}
   \,=\, \ZZ^{2}\setminus \! 
   {\textstyle \bigcup\limits_{p \in \PP}}\, p\ZZ^{2}
\] 
of the (unimodular) square lattice $\ZZ^2 \subset \RR^{2}$, where
$\PP$ denotes the set of rational primes.  A central patch of the set
$V$ is illustrated in Figure~\ref{fig:visible}. We refer the reader to
\cite{BMP,TAO,BH} for proofs of the subsequent results, which we
repeat here in an informal manner. In Section~\ref{sec:families}, we
shall discuss a substantial extension in the form of coprime
sublattice families.

It is well known that $V$ is non-periodic and has arbirarily large
holes, so it fails to be a Delone set. Nevertheless, its \emph{natural
density} exists and is equal to
\[
   \dens (V) \, = \prod_{p\in\PP}\Big(1-\frac{1}{p^2}\Big)
   \, = \, \frac{1}{\zeta(2)} \, = \, \frac{6}{\pi^2} \ts .
\]
The term `natural' refers to the use of centred, nested discs as
averaging regions; see \cite{TAO} and \cite[Appendix]{BMP} for a more
detailed discussion of this aspect. In other words, we use a van Hove
sequence $\cA$ of centred, open discs with increasing radius. Discs
can be replaced by other bodies with nice boundaries, but one has to
work with \emph{tied densities} in the sense of \cite{BMP} in order to
deal with the holes in $V$ properly.

\begin{center}
\begin{figure}
\includegraphics[width=0.84\textwidth]{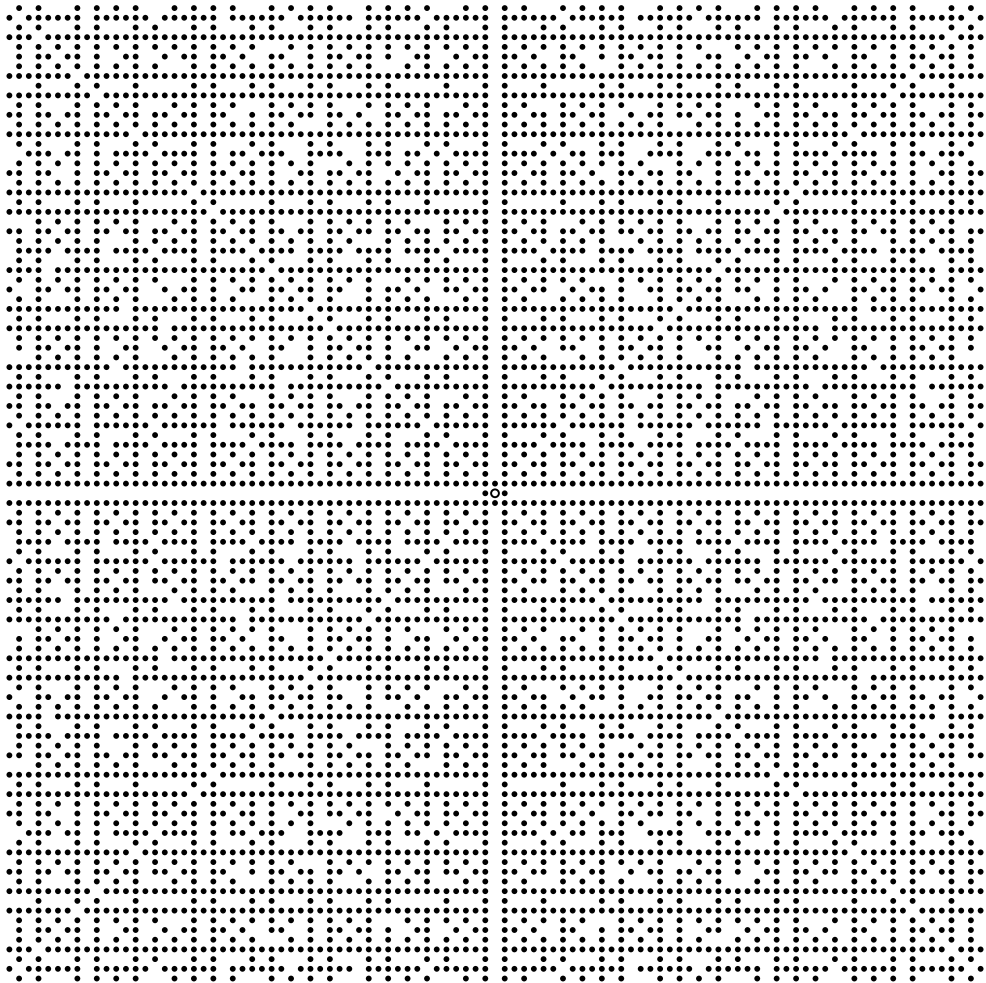}
\caption{A central patch of the visible points $V$ of the square
  lattice $\ZZ^{2}$.}
\label{fig:visible}
\end{figure}
\end{center}

Moreover, relative to $\cA$, one can explicitly compute the natural
autocorrelation measure $\gamma^{}_{V}$ together with its Fourier
transform $\widehat{\gamma^{}_{V}}$, the \emph{diffraction measure} of
$V$. It turns out that the latter is a positive pure point measure
which is translation bounded and supported on the points of $\QQ^{2}$
with square-free (s.f.) denominator, so
\[
   \widehat{\gamma^{}_{V}} \, =
   \sum_{\substack{k\in\QQ^{2}\\[0.3mm] \den (k) \text{ s.f.}}}
   I(k) \, \delta^{}_{k} \ts , \quad \text{where }
   I(k) \,=\, \biggl(\frac{6}{\pi^2}
   \prod_{p\mid\den (k)}\frac{1}{1 - p^2}\biggr)^{\! 2}.
\]
Fig.~\ref{fig:vispodiff} illustrates the diffraction measure.

\begin{center}
\begin{figure}
\includegraphics[width=0.84\textwidth]{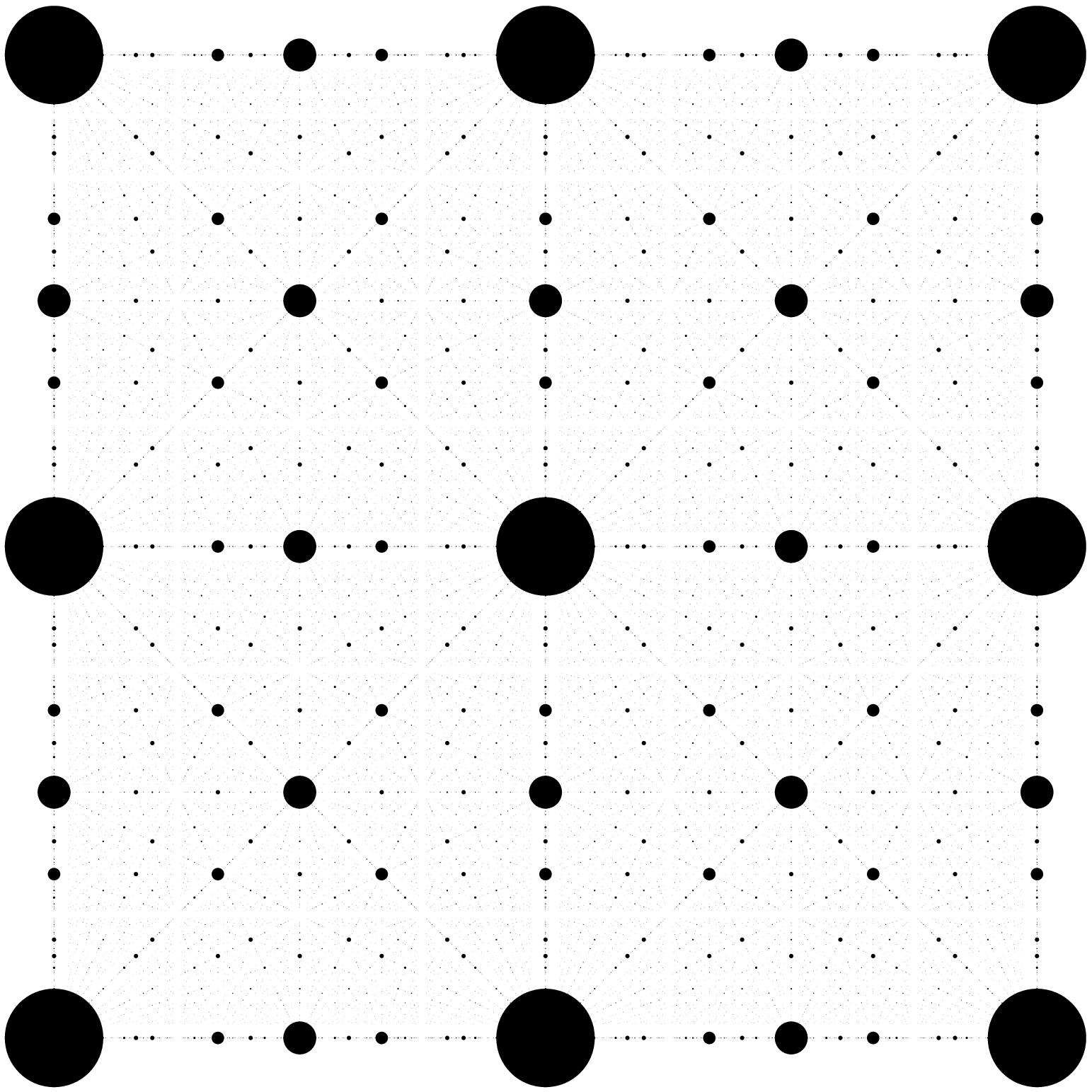}
\caption{Diffraction of the visible points of $\ZZ^{2}$. A point
  measure at $k$ with intensity $I(k)$ is shown as a disk centred at
  $k$ with area proportional to $I(k)$. Shown are the intensities with
  $I(k)/I(0)\geqslant 10^{-6}$ and $k\in [0,2]^{2}$. Its lattice of
  periods is $\ZZ^{2}$.}
\label{fig:vispodiff}
\end{figure}
\end{center}

To compare this with the \emph{dynamical spectrum}, let us define the
(discrete) hull of $V$ as
\[
   \XX_{V} \, = \, \overline{\{t+V\,\mid\,t\in\ZZ^{2}\}},
\]
with the closure being taken in the (metric) \emph{local topology},
where two subsets of $\ZZ^{2}$ are close if they agree on a large ball
around the origin. The hull $\XX_{V}$ is then compact and the
translational action of $\ZZ^{2}$ on the hull is continuous, so
$(\XX_{V},\ZZ^{2})$ is a \emph{topological dynamical system}. Since
$V$ contains holes of arbitrary size, the empty set is an element of
$\XX_{V}$. As a result, $\XX_{V}$ fails to be minimal, and the set $V$
is non-periodic, but not aperiodic in the terminology of
\cite{TAO}. Also, the hull is rather `large' in terms of the variety
of its members, unlike what one is used to from hulls of substitution
generated point sets.  Astonishingly, one can explicitly characterise
the elements of $\XX_{V}$ as the subsets of $\ZZ^{2}$ that miss at
least one coset modulo the subgroup $p\ZZ^{2}$ for any prime $p \in
\PP$ (for instance, $V$ itself misses by definition the zero coset
$0+p\ZZ^2$ modulo $p\ZZ^2$ for all $p\in\PP$).

There is a natural Borel probability measure $\nu$ on the hull
$\XX_{V}$ that originates from the natural \emph{patch frequencies} of
$V$ in space. More precisely, the frequency $\nu(\cP)$ of a
\emph{$\rho$-patch} $\cP=(V-t)\cap B_\rho(0)$ of $V$ at location $t$
(the natural density of all such $t$'s) can again be calculated
explicitly and one can then assign this very value to the
\emph{cylinder set} $C_{\cP}$ of elements $A$ of the hull with $A\cap
B_\rho(0)=\cP$. This can then uniquely be extended to a
$\ZZ^{2}$-invariant probability measure, also called $\nu$, on the
hull $\XX_{V}$, and one obtains a \emph{measure-theoretic dynamical
  system} $(\XX_{V},\ZZ^{2},\nu)$. The measure $\nu$ gives no weight
to the empty set (as a member of $\XX_{V}$), and the system becomes
aperiodic in the measure-theoretic sense of \cite[Def.~11.1]{TAO}.
  
Here, one is also interested in the \emph{dynamical spectrum}, that is
the spectrum of the corresponding unitary representation $U$ of
$\ZZ^{2}$ on the Hilbert space $L^{2}(\XX_{V},\nu)$, with the standard
inner product
\[
   \langle f \, | \, g\rangle \, =
   \int_{\XX_{V}}\overline{f(x)}\, g(x) \dd \nu (x) \ts .
\] 
The system $(\XX_{V},\ZZ^{2},\nu)$ has pure point dynamical spectrum
if and only if the eigenfunctions span all of
$L^{2}(\XX_{V},\nu)$. Since $V$ is $\nu$-generic, the individual
diffraction measure of $V$ coincides with the diffraction measure of
the system $(\XX_V,\ZZ^{2},\nu)$ in the sense of~\cite{BL,BLvE}.  By
the general equivalence theorem \cite[Thm.~7]{BL}, the pure point
nature of the dynamical spectrum follows. Moreover, the spectrum (in
additive notation) is nothing but the set of points in $\QQ^{2}$ with
square-free denominator, which form a subgroup of $\QQ^{2}$.

The set $V$ is a \emph{weak model set} in our above terminology,
originating from the CPS $(G,H,\cL)$ of Eq.~\eqref{eq:candp} with
LCAGs $G=\ZZ^{2}$, equipped with its natural discrete topology, and
$H:=\prod_{p} \ZZ^{2}\!  /\nts p\ZZ^{2}$, where $\ZZ^{2}\! / \nts
p\ZZ^{2}$ is a quotient group of order $p^2$ and $H$ is endowed with
the product topology with respect to the discrete topology on its
factors. In particular, the internal group is compact. The lattice
$\cL$ (a discrete and co-compact subgroup of $\ZZ^{2}\times H$) is
defined by the diagonal embedding of $\ZZ^{2}$,
\[ 
   \cL \,= \, \big\{ \bigl(x,\iota(x)\bigr) \mid 
    x\in\ZZ^{2}\big\} ,
\]
where $\iota (x) = (x_{p})^{}_{p\in\PP}$, with $x_{p}$ the reduction
of $x \bmod p$, is the $\star$-map in this case. In fact, one even has
$\pi (\cL) = \ZZ^{2}$ and $\pi^{}_{\mathrm{int}} (\cL) = H$ here.
With
\[
   W  :=  \prod_{p\in\PP}\bigl( \ZZ^{2}\! / \nts p \ZZ^{2}
   \setminus \{0 \}\bigr) \, \subset \, H \ts ,
\]
one clearly obtains $V=\oplam (W) \subset \ZZ^{2}$. Note that $W$ is
compact and satisfies $W=\partial W$, so has no interior. Moreover,
$W$ has positive measure $\theta^{}_{\! H} (W) = \prod_{p}
(1-\tfrac{1}{p^{2}})=1/\zeta(2)$ with respect to the normalised Haar
measure $\theta^{}_{\! H}$ on the compact group $H$. Note that
$\theta^{}_{\! H} (W)$ coincides with the natural density of $V$.

Employing the ergodicity of the frequency measure $\nu$, one can in
fact show that the dynamical system $(\XX_{V},\ZZ^2,\nu)$ is
isomorphic with the Kronecker system $(H,\ZZ^{2},\theta^{}_{\! H})$,
with the action of $\ZZ^2$ being given by addition of $\iota(x)$.
Here, we have also used the isomorphism $H \simeq (\ZZ^{2} \times H)
/\nts \cL$.  The setting of this section can be generalised to many
similar systems, all arithmetic in nature; compare \cite{BH} and
references therein.  Let us thus develop a wider scheme for weak model
sets that comprises all of them.

\section{Diffraction of weak model sets}\label{sec:diffraction}

Let us assume that a CPS $(G,H,\cL)$ is given, including some
normalisation of the Haar measures $\theta^{}_{\nts G}$ and
$\theta^{}_{\!  H}$ as indicated earlier. Since one often has to deal
with several different lattices within $G\times H$, we do \emph{not}
assume $\cL$ to be unimodular, so factors $\dens (\cL)$ will appear in
our formulas. For a single lattice $\cL$, one could rescale
$\theta^{}_{\!  H}$ to make $\cL$ unimodular, which is another
convention that is also often used.

\begin{prop}\label{prop:dens-versus-gamma}
  Let\/ $(G,H,\cL)$ be a CPS as in Eq.~\eqref{eq:candp}, with\/ $G$
  being\/ $\sigma$-compact and\/ $H$ compactly generated, and let\/
  $\varnothing\ne W\subseteq H$ be compact. Next, consider the weak
  model set\/ $\vL = \oplam (W)$ and assume that a van Hove averaging
  sequence\/ $\cA$ is given relative to which the density\/ $\dens
  (\vL)$ and the autocorrelation measure\/ $\gamma^{}_{\! \vL}$ are to
  be defined.  Then, the following statements are equivalent.
\begin{itemize}
\item[{$(1)$}] The lower density of\/ $\vL$ is maximal, 
  $ \underline{\dens} (\vL) = \dens (\cL)\, \theta^{}_{\! H} (W)$;
\item[{$(2)$}] The density of\/ $\vL$ exists and is maximal,
  $\dens (\vL) = \dens (\cL)\, \theta^{}_{\! H} (W)$;
\item[{$(3)$}] The autocorrelation of\/ $\vL$ exists and satisfies
  $\gamma^{}_{\! \vL} = \dens (\cL)\, \omega^{}_{c^{}_{W}}$.
\end{itemize}
Here, $c^{}_{W} = 1^{}_{W} * \widetilde{1^{}_{W}}$ is the
covariogram function of\/ $W$.
\end{prop}

\begin{proof}
  The equivalence of statements (1) and (2) follows from
  Fact~\ref{fact:dens-bound}.  If $\gamma^{}_{\! \vL}$ exists relative
  to $\cA$, we know that also $\dens (\vL)$ exists relative to $\cA$,
  because $\gamma^{}_{\! \vL} (\{ 0 \}) = \dens (\vL) $. Now, if
  $\gamma^{}_{\! \vL} = \dens (\cL)\, \omega^{}_{c^{}_{W}}$, one
  obtains via Eq.~\eqref{eq:cov-at-0} that
\[
   \frac{\dens (\vL)}{\dens (\cL)} \, = \,
   \frac{\gamma^{}_{\! \vL} \bigl( \{ 0 \} \bigr)}{\dens (\cL)}
   \, = \, \omega^{}_{c^{}_{W}}  \bigl( \{ 0 \} \bigr)
   \, = \, \bigl( 1^{}_{W} * \widetilde{1^{}_{W}}\bigr) (0)
   \, = \, \theta^{}_{\! H} (W) \ts ,
\]
which establishes the implication $(3) \Rightarrow (2)$.

For the converse direction, assume the existence of $\dens (\vL)$
relative to the averaging sequence $\cA$.  By
\cite[Thm.~8.13]{HR-Book}, we know that all LCAGs are normal, so we
have Urysohn's lemma at our disposal, which we now employ for an
approximation argument as follows. Here, the compactness of $W$
implies the existence of compact set $K_{g} \subseteq H$ with
$W\subset K^{\circ}_{g}$ together with a $[0,1]$-valued continuous
function $g$ with $\supp (g)\subseteq K_{g}$ that is $1$ on
$W$. Indeed, employing the regularity of the Haar measure
$\theta^{}_{\! H}$, there is even a net of $[0,1]$-valued functions
$g^{}_{\alpha} \in C_{\mathsf{c}} (H)$ with $W\subseteq \supp
(g^{}_{\alpha}) \subseteq K_{g}$, all with the same $K_{g}$, such that
$1^{}_{K_{g}} \geqslant g^{}_{\alpha} \geqslant 1^{}_{W}$ holds for
all $\alpha$ together with $\lim_{\alpha} \theta^{}_{\! H} (
g^{}_{\alpha} ) = \theta^{}_{\! H} (W)$.

Let us first observe that the choice of $(g^{}_{\alpha})$ implies
$\lim_{\alpha} \bigl(g^{}_{\alpha}\nts * \widetilde{g^{}_{\alpha}}\ts
\bigr) = 1^{}_{W} * \widetilde{1^{}_{W}}$ in $C_{\mathsf{c}} (H)$,
because, employing $\| f*h\|^{}_{\infty} \leqslant \| f\|^{}_{1} \,
\|h\|^{}_{\infty}$ with $\|1^{}_{W}\|^{}_{\infty} = \| g^{}_{\alpha}
\|^{}_{\infty} = 1$ and $ \|h\|^{}_{1} = \| \widetilde{h} \|^{}_{1} $,
one finds the estimate
\begin{equation}\label{eq:cov-estimate}
   \big\| 1^{}_{W} * \widetilde{1^{}_{W}} -
   g^{}_{\alpha}\nts\nts * \widetilde{g^{}_{\alpha}} \big\|^{}_{\infty}
    \ts \leqslant \, \big\| 1^{}_{W} * \bigl(\widetilde{1^{}_{W}} -
    \widetilde{g^{}_{\alpha}}\ts \bigr) \big\|^{}_{\infty} + \ts
   \big\|  \widetilde{g^{}_{\alpha}} *
   \bigl( 1^{}_{W} - g^{}_{\alpha} \bigr) \big\|^{}_{\infty} 
   \ts \leqslant \, 2\, \| 1^{}_{W} - g^{}_{\alpha} \|^{}_{1} \ts .
\end{equation}
Now, consider the weighted Dirac combs $\omega^{}_{\nts
  g^{}_{\alpha}\nts * \widetilde{g^{}_{\alpha}}}$, which are positive
and positive definite pure point measures, all supported within the
common model set $\oplam (K_{g}-K_{g})$, which is thus a Meyer set as
well. Since also $\supp \bigl(\omega^{}_{1^{}_{W} *
  \widetilde{1^{}_{W}}}\bigr) \subseteq \oplam (K_{g}-K_{g})$, it
follows that
\[
     \omega^{}_{\nts g^{}_{\alpha}\nts * \widetilde{g^{}_{\alpha}}}
     \;\, \xrightarrow{\quad} \;\, 
     \omega^{}_{1^{}_{W} * \widetilde{1^{}_{W}}}
     \, = \, \omega^{}_{c^{}_{W}}
\]
pointwise at each $x\in \oplam (K_{g}-K_{g})$, and hence also in the
vague topology.

To simplify the exposition, let us assume for the remainder of this
proof that $\dens (\cL) =1$, which means no restriction as we could
rescale $\theta^{}_{\! H}$ accordingly.  Let $\varepsilon > 0$ and
choose some $f\in C_{\mathsf{c}} (G)$, so $\supp (f) \subseteq K$ for
some compact set $K\subseteq G$. Since the measures
$\omega^{}_{g^{}_{\alpha}}$ and $\delta^{}_{\! \vL}$ are
equi-translation bounded, there is constant $C$ such that
\begin{equation}\label{eq:bounds}
   \|\ts \omega^{}_{g^{}_{\alpha}} \|^{}_{K} \, < \, C
   \quad \text{and} \quad
   \| \delta^{}_{\! \vL} \|^{}_{K} \, < \, C \ts .
\end{equation}
Moreover, there is some index $M$ such that $ \theta^{}_{\! H } (
g^{}_{\alpha} ) \leqslant \theta^{}_{\! H} (W) + \frac{\varepsilon} {1
  + \ts C \| f \|^{}_{\infty}} $ holds for all $\alpha \succcurlyeq
M$. We also know, possibly after adjusting the index $M$, that
\[
   \bigl|  \omega^{}_{c^{}_{W}} (f)
  -  \omega^{}_{g^{}_{\alpha} \nts *\ts \widetilde{g^{}_{\alpha}} } (f)
   \bigr| \, < \, \varepsilon
\]
holds for all $\alpha \succcurlyeq M$.  With $\vL_{n} := \vL\cap
A_{n}$, we now employ an $8\ts\varepsilon$-argument to establish the
convergence of the autocorrelation sequence $ \bigl( \frac{1}{\vol
  (A_n)}\, \delta^{}_{\!\vL_{n}} \! *
\widetilde{\delta^{}_{\!\vL_{n}}} \bigr)_{n\in\NN}$ as follows.

Let $\alpha \succcurlyeq M$ be fixed. With the abbreviation
$\omega^{}_{\alpha, n} = \omega^{}_{g^{}_{\alpha}}\nts
\big|^{}_{A_n}$, we have
\[
   \omega^{}_{g^{}_{\alpha} * \ts \widetilde{g^{}_{\alpha}}}
   \, = \, \lim_{n\to\infty} \frac{\omega^{}_{\alpha,n}\nts
   *\widetilde{\omega^{}_{\alpha,n}}}{\vol (A_n)}
\]
as a consequence of the results of \cite{RS}, see also \cite{LR} for
an alternative proof. Consequently, there is some integer $N_{1}$ so
that, for all $n\geqslant N_{1}$, we have
\[
   \biggl|\omega^{}_{g^{}_{\alpha} * \ts \widetilde{g^{}_{\alpha}}} (f)
   - \frac{\bigl(\omega^{}_{\alpha,n}\nts\nts
   *\widetilde{\omega^{}_{\alpha,n}}\bigr) (f) }{\vol (A_n)} 
   \biggr| \, < \, \varepsilon \ts .
\]

Next, recall that $\omega^{}_{g^{}_{\alpha}}$ is a norm-almost
periodic measure \cite{BM,RS}, so it is amenable with the mean being
given by $\lim_{n\to\infty} \frac{1}{\vol (A_{n})}\,
\omega^{}_{g^{}_{\alpha}}\! (A_{n}) = \theta^{}_{\! H}
(g^{}_{\alpha})$.  There is thus some $N_{2}$ so that $ \theta^{}_{\!
  H} (g^{}_{\alpha}) \geqslant \frac{1}{\vol (A_{n})}\,
\omega^{}_{g^{}_{\alpha}}\! (A_{n}) - \frac{\varepsilon} {1+\ts C \| f
  \|^{}_{\infty}}$ holds for all $n > N_{2}$. Finally, due to our
density condition for $\vL = \oplam (W)$, there exists some $N_{3}$
such that $\frac{1}{\vol (A_{n})}\, \delta^{}_{\! \vL} (A_{n}) >
\theta^{}_{\! H} (W) - \frac{\varepsilon} {1+\ts C \| f
  \|^{}_{\infty}}$ is satisfied for all $n > N_{3}$.

Now, if $n > \max \{ N_{1}, N_{2}, N_{3} \}$, we have
\[
   \frac{\delta^{}_{\! \vL} (A_{n}) }{\vol (A_{n})}
   \, > \, \theta^{}_{\! H} (W) -  \frac{\varepsilon} 
    {1+\ts C \| f \|^{}_{\infty}} \, \geqslant \,
   \theta^{}_{\! H} (g^{}_{\alpha}) -  \frac{2 \ts \varepsilon} 
    {1+\ts C \| f \|^{}_{\infty}} \, \geqslant \,
   \frac{\omega^{}_{g^{}_{\alpha}}\! (A_{n})}{\vol (A_{n})}
   - \frac{3 \ts \varepsilon}  {1+\ts C \| f \|^{}_{\infty}} 
\]
and hence $\omega^{}_{g^{}_{\alpha}}\! (A_{n}) - \delta^{}_{\! \vL}
(A_{n}) \, < \, \frac{3 \ts \varepsilon\, \vol (A_{n})}{1+\ts
  C \| f \|^{}_{\infty}}$.  As $\omega^{}_{g^{}_{\alpha}} \geqslant
\delta^{}_{\! \vL}$, the finite measure
$\frac{1}{\vol (A_{n})} \bigl( \omega^{}_{\alpha,n}
-\delta^{}_{\! \vL_{n}}\bigr)$ is positive. Consequently,
\[
   \big\| \ts \omega^{}_{\alpha,n} -\delta^{}_{\! \vL_{n}}
   \big \| \, = \,
   \bigl(\omega^{}_{\alpha,n} -\delta^{}_{\! \vL_{n}}\bigr) (G)
   \, = \, \bigl(\omega^{}_{g^{}_{\alpha}}\! -
   \delta^{}_{\! \vL}\bigr) (A_{n}) \, < \,
   \frac{3 \ts \varepsilon\, \vol (A_{n})}
    {1+\ts C \| f \|^{}_{\infty}} \ts .
\]
Clearly, this also implies $\big\| \ts \widetilde{\omega^{}_{\alpha,n}}
-\widetilde{\delta^{}_{\! \vL_{n}}} \big \| < \frac{3 \ts
  \varepsilon\, \vol (A_{n})} {1+\ts C \| f \|^{}_{\infty}}$.  Put
together, we can now estimate
\[
\begin{split}
  \biggl| \omega^{}_{c^{}_{W}} (f) - 
  \frac{ \bigl( \delta^{}_{\! \vL_{n}} *
  \widetilde{\delta^{}_{\! \vL_{n}}} \bigr) (f)}
   {\vol (A_{n})} \biggr| 
  \, & \leqslant \, \bigl|  \omega^{}_{c^{}_{W}} (f)
  -  \omega^{}_{g^{}_{\alpha} \nts *\ts \widetilde{g^{}_{\alpha}} } (f)
   \bigr| + \biggl| \ts
  \omega^{}_{g^{}_{\alpha} \nts *\ts \widetilde{g^{}_{\alpha}} } (f) -
  \frac{ \bigl( \omega^{}_{\alpha,n} *
  \widetilde{\omega^{}_{\alpha,n}} \bigr) (f)}
  {\vol (A_{n})} \biggr|  \\[2mm]
 & \qquad\qquad + \biggl|  \frac{ \bigl( \omega^{}_{\alpha,n} *
  \widetilde{\omega^{}_{\alpha,n}} \bigr) (f)}
  {\vol (A_{n})} -
   \frac{ \bigl( \delta^{}_{\! \vL_{n}} *
  \widetilde{\delta^{}_{\! \vL_{n}}} \bigr) (f)}
   {\vol (A_{n})} \biggr| \\[2mm]
& < \,  \biggl|  \frac{ \bigl( \omega^{}_{\alpha,n} *
  \widetilde{\omega^{}_{\alpha,n}} \bigr) (f)}
  {\vol (A_{n})} -
   \frac{ \bigl( \delta^{}_{\! \vL_{n}} *
  \widetilde{\delta^{}_{\! \vL_{n}}} \bigr) (f)}
   {\vol (A_{n})} \biggr| + 2\ts \varepsilon \\[2mm]
& \leqslant \, \biggl| \frac{ \bigl( (\omega^{}_{\alpha,n} -
  \delta^{}_{\! \vL_{n}}) * \widetilde{\omega^{}_{\alpha,n}} \bigr) (f)}
  {\vol (A_{n})} \biggr| + \biggl|
   \frac{ \bigl( \delta^{}_{\! \vL_{n}} *
  \bigl( \widetilde{\omega^{}_{\alpha,n}} -
   \widetilde{\delta^{}_{\! \vL_{n}}} \bigr) \bigr) (f)}
  {\vol (A_{n})} \biggr| + 2\ts \varepsilon \\[2mm]
& \leqslant \,  \bigg\| \frac{ \omega^{}_{\alpha,n} -
  \delta^{}_{\! \vL_{n}} } {\vol (A_{n})} \bigg\| \,
  \big\| \widetilde{\omega^{}_{\alpha,n}} \big\|_{K} 
  \, \| f \|^{}_{\infty} + 
  \big\| \delta^{}_{\! \vL_{n}} \big\|_{K}  \,
  \bigg\| \frac{  \widetilde{\omega^{}_{\alpha,n}} -
  \widetilde{\delta^{}_{\! \vL_{n}}} } 
  {\vol (A_{n})} \bigg\| \,
  \, \| f \|^{}_{\infty} + 2\ts \varepsilon \\[2mm]
& \leqslant \, \frac{6\ts \varepsilon}{1 + \ts C \|f\|^{}_{\infty}}
 \, C \|f\|^{}_{\infty} + 2\ts \varepsilon
  \, < \, 8\ts \varepsilon \ts ,
\end{split}
\]
where the second last line follows from Fact~\ref{fact:bounds} and the
last from Eq.~\eqref{eq:bounds}. This completes our argument.
\end{proof}

\begin{remark}
  Let us note that the estimate in Eq.~\eqref{eq:cov-estimate} implies
  the uniform convergence of the net $(g^{}_{\alpha}\nts\nts *
  \widetilde{g^{}_{\alpha}})$ of continuous functions to the limit
  $c^{}_{W}$.  The latter must then be contiunous, too, in line with
  Fact~\ref{fact:cont}. Since this type of argument can also be used
  for measurable sets $W$ that are merely relatively compact, one sees
  the continuity of $c^{}_{W}$ in an alternative way.  \exend
\end{remark}

\begin{remark}\label{rem:mixed}
  In the setting of Proposition~\ref{prop:dens-versus-gamma}, let
  $\vL^{}_{1} = \oplam (W^{}_{\! 1})$ and $\vL^{}_{2} = \oplam
  (W^{}_{\! 2})$ be two weak model sets with compact windows such
  that, for the same van Hove sequence $\cA$, we have
\[
     \dens (\vL_{i}) \, = \, \dens (\cL) \,
     \theta^{}_{\! H} (W^{}_{\! i}) \ts ,
    \quad \text{for } i\in \{1,2\} \ts .
\]
Then, in complete analogy to the proof of
Proposition~\ref{prop:dens-versus-gamma}, one can show that
\[
  \lim_{n\to\infty} \frac{\delta^{}_{\vL^{}_{1} \cap\ts A^{}_{n}} 
  \! * \widetilde{\delta^{}_{\vL^{}_{2} \cap\ts A^{}_{n}}}}
   {\vol (A_{n})} \, = \, \dens (\cL) \,
   \omega^{}_{1^{}_{W^{}_{\! 1}}\! *\ts
    \widetilde{1^{}_{W^{}_{\! 2}}}}
\]
holds for the mixed correlation (or Eberlein convolution) between
$\vL^{}_{1}$ and $\vL^{}_{2}$.  \exend
\end{remark}

Let us extend our previous result to more general windows, namely to
relatively compact sets $W\subset H$ with $\theta^{}_{\! H}
(\,\overline{\! W\!}\,) >0$.  One then finds the following result.

\begin{coro}\label{coro:dens-versus-gamma}
  Consider the setting of Proposition~$\ref{prop:dens-versus-gamma}$,
  with van Hove averaging sequence\/ $\cA$ and\/ $\vL = \oplam (W)$,
  but only assume\/ $W \subset H$ to be relatively compact. If\/
  $\underline{\dens} (\vL) = \dens (\cL)\, \theta^{}_{\! H}
  (\,\overline{\! W \! } \, )$, or equivalently\/ $\dens (\vL) = \dens
  (\cL)\, \theta^{}_{\! H} (\,\overline{\! W \! } \, )$, one also
  has\/ $\gamma^{}_{\! \vL} = \dens (\cL) \,
  \omega^{}_{c^{}_{\,\overline{\! W \!}}}$, where\/
  $c^{}_{\,\overline{\! W \!}}$ is the covariogram function of\/
  $1^{}_{\,\overline{\! W \!}\,}$.
\end{coro}

\begin{proof}
  The two conditions are equivalent by Fact~\ref{fact:dens-bound};
  compare the proof of Proposition~\ref{prop:dens-versus-gamma}.  With
  respect to $\cA$, we find via Eq.~\eqref{eq:lower-upper-dens} that
\[
    \dens (\cL) \, \theta^{}_{\! H} (\,\overline{\! W\!} \, )
     \, = \, \dens (\vL)
    \, \leqslant \, \underline{\dens}
     \bigl( \oplam (\, \overline{\! W \!} \,) \bigr)
    \, \leqslant \, \overline{\dens} \bigl( 
    \oplam (\, \overline{\! W \!} \,) \bigr)  
    \, \leqslant \, \dens (\cL) \, \theta^{}_{\! H}
    (\, \overline{\! W \!}\,) \ts .
\]
This implies $\underline{\dens} \bigl( \oplam (\, \overline{\! W \!}
\,) \bigr) = \overline{\dens} \bigl( \oplam (\, \overline{\! W \!} \,)
\bigr)$ and thus the existence of the limit
\[
      \dens \bigl( \oplam (\, \overline{\! W \!} \,) \bigr) \, = 
      \lim_{n\to\infty} \frac{\card \bigl(\oplam(\, \overline{\! W \!} \,) 
           \cap A_{n} \bigr)}{\vol (A_{n})}  \, = \,
       \dens (\cL) \, \theta^{}_{\! H} (\, \overline{\! W \! }\,) \ts .    
\]
As $\vL \subseteq \oplam (\, \overline{\! W \!} \,)$, the argument
also implies that the point set $\oplam (\, \overline{\! W \!}\,)
\setminus \vL$ has zero density. Consequently, due to the inclusion
relation, $\vL = \oplam (W)$ and $\oplam (\, \overline{\! W\!} \,)$
possess the same autocorrelation measure relative to $\cA$, so
$\gamma^{}_{\! \vL} = \gamma^{}_{\! \mbox{\normalsize $\curlywedge$}
  (\, \overline{\! W \!}\,)}$.  The claim now follows from
Proposition~\ref{prop:dens-versus-gamma}.
\end{proof}

Our derivation so far motivates the following concept.

\begin{definition}\label{def:one}
  For a given CPS $(G,H,\cL)$ with $G$ $\sigma$-compact and $H$
  compactly generated, a projection set $\oplam (W)$ is called a
  \emph{weak model set of maximal density} relative to a given van
  Hove averaging sequence $\cA$ if the window $W\subseteq H$ is
  relatively compact with $\theta^{}_{\! H} (\, \overline{\! W \!} \,)
  > 0$, if the density of $\oplam (W)$ relative to $\cA$ exists, and
  if the density condition $\dens (\oplam (W)) = \dens (\cL) \,
  \theta^{}_{\!  H} (\, \overline{\! W \!}  \,)$ is satisfied.
\end{definition}

Note that, in view of Fact~\ref{fact:dens-bound}, the two conditions
on the density can be replaced by the single maximality condition
$\underline{\dens} (\oplam (W)) = \dens (\cL) \, \theta^{}_{\!  H} (\,
\overline{\! W \!}  \,)$, which is equivalent.  If $\, \overline{\! W
  \!} \,$ has zero measure in $H$, the corresponding projection set
would have upper density zero, and is not of interest in our setting.
Thus, we formulate our general diffraction result as follows.

\begin{theorem}\label{thm:weak-pp}
  Let\/ $\vL = \oplam (W)$, with\/ $W\subseteq H$ compact and\/
  $\theta^{}_{\! H} (W) >0$, be a weak model set of maximal density
  for the CPS\/ $(G,H,\cL)$, in the setting of
  Proposition~$\ref{prop:dens-versus-gamma}$.  Then, the
  autocorrelation\/ $\gamma^{}_{\! \vL}$ is a strongly almost periodic
  pure point measure. It is Fourier transformable, and\/
  $\widehat{\ts\gamma^{}_{\! \vL}\ts}$ is a translation bounded,
  positive, pure point measure on\/ $\widehat{G}$. It is explicitly
  given by
\[
     \widehat{\ts\gamma^{}_{\! \vL}\ts} \, = 
     \sum_{u\in L^{0}} \big\lvert a (u) \big\rvert^{2}\, 
     \delta^{}_{u} \, , \quad \text{with amplitude  } 
     a (u) \, = \, \frac{\dens (\vL)}{\theta^{}_{\! H} (W)}\, 
     \widehat{1^{}_{W}} (-u^{\star}) \ts ,
\]
where\/ $\widehat {1^{}_{W}}$ is a bounded, continuous function on the
dual group\/ $\widehat{H}$ and\/ $L^{0} = \pi (\cL^{\ts 0}) \subset
\widehat{G}$ is the corresponding Fourier module in additive notation.

More generally, if\/ $W\!$ is relatively compact with\/ $\theta^{}_{\!
  H} (W) >0$, but\/ $\dens (\vL) = \dens (\cL)\, \theta^{}_{\! H}
(\,\overline{\! W \! } \, )$ as in
Corollary~$\ref{coro:dens-versus-gamma}$, the previous formula holds
with\/ $W\!$ replaced by\/ $\,\overline{\! W \! }$.
\end{theorem}

\begin{proof}
  Note that, under our assumptions, we have $\dens (\cL) = \dens
  (\vL)\, \theta^{}_{\! H} (\,\overline{\! W \!} \,)$ in both
  cases. If $W$ is compact, the density assumption gives us
  $\gamma^{}_{\! \vL} = \dens (\cL) \,\omega^{}_{c^{}_{W}}$ by an
  application of Proposition~\ref{prop:dens-versus-gamma}. Since
  $c^{}_{W} = 1^{}_{W} * \widetilde{1^{}_{W}}$ is continuous (by
  Fact~\ref{fact:cont}), compactly supported (because $\supp
  (c^{}_{W}) \subseteq W\!\nts\nts-\! W$) and positive definite by
  construction, we can invoke Theorem~\ref{thm:diffraction-formula},
  which proves the claim on $\widehat{\ts\gamma^{}_{\! \vL}\ts}$.

  The extension to a relatively compact window $W\!$, under our
  density assumption, is a consequence of
  Corollary~\ref{coro:dens-versus-gamma}.
\end{proof}

\begin{remark}\label{rem:regular}
  A regular model set with proper window $W$ is automatically a weak
  model set of maximal density, and thus a special case of
  Theorem~\ref{thm:weak-pp}. Also, if $W$ is regular, but fails to be
  proper, we still get a special case when $\theta^{}_{\! H}
  (W^{\circ}) = \theta^{}_{\! H} (\, \overline{\! W \!} \,)$ is
  satisfied.  \exend
\end{remark}

Let us note in passing that the explicit formula for the diffraction
measure is nice and systematic, and a direct generalisation of the
known formula for regular model sets; compare \cite[Thm.~9.4]{TAO}.
In particular, we still have a well-defined meaning of the intensities
in terms of an amplitude (or Fourier--Bohr coefficient), which will be
useful for calculations with superpositions of Dirac combs.  However,
the result is also a bit deceptive in the sense that it will be
difficult to actually calculate the amplitudes $a (u)$ in this
generality --- unless one has an underlying arithmetic structure as in
our guiding example.

In fact, the analogy with the properties of regular model sets goes
further, via the following result on the amplitudes (or Fourier--Bohr
coefficients), which resembles the original result in \cite{Hof}, but
cannot be proved by the methods of that paper due to the absence of
uniform densities. Recall that, due to our additive notation of the
dual groups, the character $\chi^{}_{u}$ defined by $u\in\widehat{G}$
is written as $\chi^{}_{u} (\cdot) = \langle u, \cdot\rangle$, with
$\langle -u, \cdot\rangle = \overline{\langle u, \cdot\rangle}$.

\begin{prop}\label{prop:amplitudes}
  Under the conditions of Theorem~$\ref{thm:weak-pp}$, the limit
\[
   a^{}_{u} \, := \lim_{n\to\infty} \frac{1}{\vol (A_{n})}
    \sum_{t \in \vL\cap A_{n}} \! \! \overline{\langle u,t \rangle} \, = 
    \lim_{n\to\infty} \frac{1}{\vol (A_{n})}
    \int_{\nts A_{n}}\! \overline{\langle u,t\rangle} 
    \dd \delta^{}_{\! \vL} (t)
\]
  exists for each\/ $u\in\widehat{G}$, and one has\/
\[
    \widehat{\ts \gamma^{}_{\! \vL}\ts} (\{ u \} ) 
    \, = \, \lvert a^{}_{u} \rvert^{2} .
\]
Moreover, $a^{}_{u} = 0$ if\/ $u\notin L^{0}$ with\/ $L^{0}$ from the
dual CPS according to \eqref{eq:dual-candp}, and\/ $a^{}_{u}$ agrees
with the amplitude\/ $a(u)$ of Theorem~$\ref{thm:weak-pp}$ otherwise.
\end{prop}

\begin{proof}
  The proof is methodically similar to the one of
  Proposition~\ref{prop:dens-versus-gamma}. Let $u\in\widehat{G}$ be
  fixed. For each $\varepsilon > 0$, there is some index $M$ and some
  integer $N$ such that
\[
    0 \, \leqslant \, \omega^{}_{g^{}_{\alpha}}
    (A_{n}) - \delta^{}_{\! \vL} (A_{n}) \, < \,
    \varepsilon \ts \vol (A_{n})
\]
holds for $\alpha = M$ and all $n\succ N$, and actually also for all
$\alpha \succ M$, due to the monotonicity of the functions
$g^{}_{\alpha}$. Consequently, for all $n>N$ and $\alpha \succ M$, we
have
\begin{equation}\label{eq:help-1}
   \biggl| \frac{1}{\vol (A_{n})} \int_{\nts A_{n}} 
   \!\overline{\langle u, t\rangle} \dd 
   \bigl(\omega^{}_{\alpha} - \delta^{}_{\! \vL}\bigr) (t)
   \biggr| \, < \, \varepsilon \ts .
\end{equation}

Now, fix some $\alpha\succ M$ and observe that
$\omega^{}_{g^{}_{\alpha}}$ is norm-almost periodic, so the mean
\[
   \MM \bigl(\ts \overline{\langle u, \cdot\rangle} \, 
   \omega^{}_{g^{}_{\alpha}} \bigr) \, = 
   \lim_{n\to\infty} \frac{1}{\vol (A_{n})}
   \int_{\nts A_{n}} \!\overline{\langle u ,t\rangle} \dd 
   \omega^{}_{g^{}_{\alpha}} (t)
\]
exists and satisfies $ \widehat{\omega^{}_{\! g^{}_{\alpha} \nts *
    \widetilde{g^{}_{\alpha}}}} \bigl( \{ u \} \bigr) = \bigl| \MM
\bigl(\ts \overline{\langle u, \cdot \rangle} \,
\omega^{}_{g^{}_{\alpha}}\bigr)\bigr|^{2}$. Now, if $u\notin L^{0}$,
we have $ \widehat{\omega^{}_{\! g^{}_{\alpha} \nts *
    \widetilde{g^{}_{\alpha}}}} \bigl( \{ u \} \bigr) = 0$ as a
consequence of \cite{LR}.  Thus, there is some integer $N'$ such that
\[
  \biggl| \frac{1}{\vol (A_{n})} \int_{\nts A_{n}} 
  \!\overline{\langle u , t \rangle} 
  \dd  \omega^{}_{\alpha} (t) \biggr|^{2} \, < \, \varepsilon^{2}
\]
holds for all $n>N'$. Together with Eq.~\eqref{eq:help-1}, this
gives
\[
  \biggl| \frac{1}{\vol (A_{n})} \int_{\nts A_{n}} 
  \!\overline{\langle u , t \rangle} 
  \dd  \delta^{}_{\! \vL} (t) \biggr| \, < \, 2 \ts \varepsilon
\]
for all $n>N'$, and thus $a^{}_{u} = 0$ in this case.

If $u\in L^{0}$, we have $ \widehat{\omega^{}_{\! g^{}_{\alpha} \nts *
    \widetilde{g^{}_{\alpha}}}} \bigl( \{ u \} \bigr) = \lvert
\widehat{g^{}_{\alpha}} \rvert^{2} (-u^{\star})$ by \cite{LR}.  Due to
the choice of the net $(g^{}_{\alpha})$, we know that
$\widehat{1^{}_{W}} (v) = \lim_{\alpha} \widehat{g^{}_{\alpha}} (v)$
for all $v\in H$, hence $\bigl| \widehat{1^{}_{W}} (-u^{\star}) -
\widehat{g^{}_{\alpha}} (-u^{\star}) \bigr| < \varepsilon$ for some
$\alpha \succ M$.  As $ \MM \bigl(\ts \overline{\langle u, \cdot
  \rangle} \, \omega^{}_{g^{}_{\alpha}}\bigr) =
\widehat{g^{}_{\alpha}} (-u^{\star})$, there is some $N_{1} > N$ such
that, for all $n>N_{1}$, one has
\[
   \biggl| \ts \widehat{g^{}_{\alpha}} (-u^{\star}) -
    \frac{1}{\vol (A_{n})} \int_{\nts A_{n}} 
    \! \overline{\langle u , t \rangle} \dd 
    \omega^{}_{\alpha} (t) \biggr| \, < \, \varepsilon \ts .
\]   
Combined with Eq.~\eqref{eq:help-1} again, we get
\[
   \biggl| \widehat{1^{}_{W}} (-u^{\star}) -
    \frac{1}{\vol (A_{n})} \int_{\nts A_{n}} 
    \! \overline{\langle u , t \rangle} \dd 
    \omega^{}_{\alpha} (t) \biggr| \, < \, 3\ts \varepsilon
\]
for all $N>N_{1}$. This gives $a^{}_{u} = \widehat{1^{}_{W}}
(-u^{\star})$ as claimed.
\end{proof}

At this point, one should realise that relatively little can be said
in general when the maximality condition is not satisfied, as becomes
evident from the large class of Toeplitz sequences. They can be
realised as model sets with proper windows \cite{BLO}, hence also as
weak model sets. The irregular Toeplitz sequences then display a
wealth of possible phenomena. However, there is still one other class
of weak model sets which behaves nicely, too.

\begin{definition}\label{def:two}
  For a given CPS $(G,H,\cL)$ as in Definition~\ref{def:one}, a
  projection set $\oplam (W)$ is called a \emph{weak model set of
    minimal density} relative to a given van Hove averaging sequence
  $\cA$ if the window $W\subseteq H$ is relatively compact with
  $\theta^{}_{\! H} ( W^{\circ}) > 0$, if the density of $\oplam (W)$
  relative to $\cA$ exists, and if the density condition $\dens
  (\oplam (W)) = \dens (\cL) \, \theta^{}_{\!  H} ( W^{\circ})$ is
  satisfied.
\end{definition}

Note first that weak model sets of minimal density, in this setting,
are Meyer sets, and thus perhaps less interesting than their
counterparts with maximal density. In analogy to before, we could
alternatively demand $\overline{\dens} (\oplam (W)) = \dens (\cL) \,
\theta^{}_{\!  H} ( W^{\circ})$, which then entails the existence of
the density via Fact~\ref{fact:dens-bound}. Also, one could extend the
setting by only asking for $W^{\circ}$ to be relatively compact, not
$W$, which might take us outside the Meyer set class; we will not
pursue this further here. Nevertheless, repeating our approximation
with weighted Dirac combs $\omega^{}_{h^{}_{\alpha}}$, this time from
below via a suitable net $(h^{}_{\alpha})$ of compactly supported
functions on $H$ such that $0 \leqslant \omega^{}_{h^{}_{\alpha}}
\leqslant \delta^{}_{\! \smoplam (W^{\circ})}$ holds, one finds the
following analogue of Theorem~\ref{thm:weak-pp}.

\begin{theorem}\label{thm:min-pp}
  Let\/ $(G,H,\cL)$ be a CPS as in Theorem~$\ref{thm:weak-pp}$ and
  let\/ $W\subseteq H$ be relatively compact with\/ $\theta^{}_{\! H}
  (W^{\circ}) >0$.  Consider the weak model set\/ $\vL = \oplam (W)$
  and assume that it is of minimal density for a given van Hove
  averaging sequence\/ $\cA$ in\/ $G$, in the sense of
  Definition~$\ref{def:two}$.  Then, the autocorrelation measure of\/
  $\vL$ resp.\ $\delta^{}_{\! \vL}$ relative to\/ $\cA$ exists and
  satisfies\/ $\gamma^{}_{\!  \vL} = \dens (\cL) \,
  \omega^{}_{c^{}_{W^{\circ}}}$, where\/ $c^{}_{W^{\circ}}$ is the
  covariogram function of\/ $1^{}_{W^{\circ}}$. Moreover,
  $\widehat{\ts\gamma^{}_{\! \vL}\ts}$ is a translation bounded,
  positive, pure point measure, which is given by the formulas from
  Theorem~$\ref{thm:weak-pp}$ with\/ $W$ replaced by\/ $W^{\circ}$.
  Finally, the statement of Proposition~$\ref{prop:amplitudes}$ holds
  here as well.
\end{theorem}

\begin{proof}
  Instead of repeating our previous arguments with an approximation
  from below, let us employ an alternative argument that emphasises
  the complementarity. First, under our minimality assumption, the
  weak model sets $\vL=\oplam (W)$ and $\vL^{}_{0}:=\oplam
  (W^{\circ})$ possess the same autocorrelation measure (provided the
  limit along $\cA$ exists, which we still have to show), as $\oplam
  (W\setminus W^{\circ})$ is a point set of zero density for $\cA$.

  Now, let $K\subseteq H$ be a compact set that contains $W\!$, is
  proper, and satisfies $\theta^{}_{\! H} (\partial K) =0$, which is
  clearly possible. Consequently, $\vL^{}_{1} := \oplam (K)$ is a
  regular model set with all the nice properties, in particular
  relative to $\cA$. It is thus also a weak model set of maximal
  density for $\cA$.

  Next, consider $\vL^{}_{2} := \oplam (K\setminus W^{\circ})$, which
  is another weak model set of maximal density relative to $\cA$.  As
  $\delta^{}_{\! \vL^{}_{0}} = \delta^{}_{\! \vL^{}_{1}} -
  \delta^{}_{\!  \vL^{}_{2}}$, we can employ Remark~\ref{rem:mixed} to
  relate the various correlation measures. With $\vL_{i,n} :=
  \vL_{i}\cap\ts A_{n}$, we get for the approximating autocorrelations
  of $\delta^{}_{\! \vL^{}_{0}}$ that
\[
\begin{split}
  \lim_{n\to\infty} \frac{\delta^{}_{\! \vL^{}_{0,n}}\! * 
  \widetilde{\delta^{}_{\! \vL^{}_{0,n}}}}
  {\vol (A_{n})} \, & = \lim_{n\to\infty} \frac{ 
  \bigl( \delta^{}_{\! \vL^{}_{1,n}} \! - 
     \delta^{}_{\! \vL^{}_{2,n}} \bigr) 
  * \bigl( \delta^{}_{\! \vL^{}_{1,n}}\! -
    \delta^{}_{\! \vL^{}_{2,n}} \bigr)^{\! \widetilde{\phantom{aa}}}} 
    {\vol (A_{n})} \\
  & =   \lim_{n\to\infty} \left(
  \frac{\delta^{}_{\! \vL^{}_{1,n}}\! * 
   \widetilde{\delta^{}_{\! \vL^{}_{1,n}}}}
  {\vol (A_{n})} + \frac{\delta^{}_{\! \vL^{}_{2,n}}\! *
   \widetilde{\delta^{}_{\! \vL^{}_{2,n}}}} {\vol (A_{n})} 
  - \frac{\delta^{}_{\! \vL^{}_{2,n}}\! * 
   \widetilde{\delta^{}_{\! \vL^{}_{1,n}}}}{\vol (A_{n})}
   - \frac{\delta^{}_{\! \vL^{}_{1,n}}\! *
   \widetilde{\delta^{}_{\! \vL^{}_{2,n}}}}
  {\vol (A_{n})}  \right) \\[2mm]
  & = \, \dens (\cL) \Bigl(
  \omega^{}_{1^{}_{\! K} * \widetilde{1^{}_{\! K}}} + 
  \omega^{}_{1^{}_{\! K\setminus W^{\circ}} *
    \widetilde{1^{}_{\! K\setminus W^{\circ}}}} -
  \omega^{}_{1^{}_{\! K\setminus W^{\circ}} * \widetilde{1^{}_{\! K}}} -
  \omega^{}_{1^{}_{\! K} * \widetilde{1^{}_{\! K\setminus W^{\circ}}}}
   \Bigr) \\[1mm]
  & = \, \dens (\cL) \, \omega^{}_{
  (1^{}_{\! K} - 1^{}_{\! K\setminus W^{\circ}}) *
  (1^{}_{\! K} - 1^{}_{\! K\setminus W^{\circ}})^{\!\widetilde{\phantom{aa}}}}
  \, = \,  \dens (\cL) \, \omega^{}_{c^{}_{W^{\circ}}} \ts ,
\end{split}
\]
which also establishes the existence of the limit and thus completes
our argument.
\end{proof}

At this stage, in the spirit of Fact~\ref{fact:dens-bound}, one can
derive the following sandwich result for an arbitrary autocorrelation
of a weak model set.

\begin{coro}\label{coro:sandwich}
  Let\/ $\vL$ be a weak model set for the CPS\/ $(G,H,\cL)$ from
  above, with relatively compact window\/ $W\!$. If\/ $\gamma$ is any
  autocorrelation of\/ $\vL$, it satisfies the measure inequality
\[
     0 \, \leqslant \,
     \dens (\cL) \, \omega^{}_{c^{}_{W^{\circ}}} \leqslant \,
     \gamma \, \leqslant \, \dens (\cL) \,
     \omega^{}_{c^{}_{\,\overline{\! W\! }\,}} \ts ,
\]   
with\/ $c^{}_{\! A}$ the covariogram function of the measurable set\/
$A$.
\end{coro}

\begin{proof}
  Select nets $(h^{}_{\alpha})$ and $(g^{}_{\alpha})$ of continuous
  functions with $h^{}_{\alpha} \, {\scriptstyle \nearrow} \,
  1^{}_{W^{\circ}}$ and $g^{}_{\alpha} \, {\scriptstyle \searrow} \,
  1^{}_{\,\overline{\! W \!}\,}$ in analogy to our previous
  arguments. Also, let $\cB = (B_{n})^{}_{n\in\NN}$ be a van Hove
  sequence relative to which the autocorrelation of $\vL$ is
  $\gamma$. Then, as $\omega^{}_{h^{}_{\alpha}}$ and
  $\omega^{}_{g^{}_{\alpha}}$ are norm-almost periodic measures, their
  autocorrelation measures with respect to $\cB$ exist and are given
  by $\omega^{}_{\! h^{}_{\alpha}\nts * \widetilde{h^{}_{\alpha}}}$
  and $\omega^{}_{\nts g^{}_{\alpha} \nts * \ts
    \widetilde{g^{}_{\alpha}}}$, respectively.  Both are positive and
  positive definite measures.

  Now, we have $0 \leqslant \omega^{}_{h^{}_{\alpha}} \leqslant
  \delta^{}_{\!\vL} \leqslant \omega^{}_{g^{}_{\alpha}}$ by
  construction, which implies
\[
    0 \, \leqslant \,
    \dens (\cL) \, \omega^{}_{h^{}_{\alpha}\nts * 
    \widetilde{h^{}_{\alpha}}} \, \leqslant \, \gamma
    \, \leqslant \, \dens (\cL) \,
    \omega^{}_{g^{}_{\alpha} \nts * \ts \widetilde{g^{}_{\alpha}}}
\]
by standard arguments. Since $h^{}_{\alpha}\nts *
\widetilde{h^{}_{\alpha}} \xrightarrow{\quad} c^{}_{W^{\circ}}$ and
$g^{}_{\alpha} \nts * \ts \widetilde{g^{}_{\alpha}}
\xrightarrow{\quad} c^{}_{\, \overline{\! W \!}\,}$, we obtain the
claimed inequality by taking the limits of the previous inequality in
$\alpha$.
\end{proof}

The general spectral theory can now be developed further, aiming at a
result on the dynamical spectrum of the hull of weak model sets of
extremal density. For this, we first need to construct a suitable
measure and establish its ergodicity.

\section{Hull, ergodicity and dynamical spectrum}\label{sec:hull}

Let us fix a CPS $(G,H,\cL)$ with a $\sigma$-compact LCAG $G$, a
compactly generated LCAG $H$ and a lattice $\cL \subset G \times H$ as
before, and let $\vL = \oplam (W)$ with compact $W\subseteq H$ be a
weak model set of maximal density, relative to a fixed van Hove
averaging sequence $\cA = (A_{n})^{}_{n\in \NN}$. The (geometric) hull
of $\vL$ is the orbit closure $\overline{G + \nts\vL}$ in the local
topology; compare \cite[Sec.~5.4]{TAO} for background. Note that our
point set $\vL$ is an FLC set, so that the local topology suffices (it
is a special case of a Fell topology \cite{BL}). The group $G$ acts
continuously on the hull by translations.

In view of our further reasoning, we now represent $\vL$ by its Dirac
comb $\delta^{}_{\! \vL}$, which is a translation bounded, positive
pure point measure with support $\vL$. Its hull is
\[
     \XX^{}_{\vL} \, := \, \overline{ \{ \delta^{}_{t} * 
    \delta^{}_{\! \vL} \mid t \in G \} }\ts ,
\]
where the closure is taken in the vague topology. By standard
arguments, $\XX^{}_{\vL}$ is vaguely compact, with a continuous action
of $G$ on it. Clearly, $\delta^{}_{t} * \delta^{}_{\! \vL} =
\delta^{}_{t+\vL}$, so that the topological dynamical systems $(\vL,
G)$ and $(\XX^{}_{\vL},G)$ are topologically conjugate.

Let now $(g^{}_{\alpha})$ be the net of $C_{\mathsf{c}} (H)$-functions
with $1^{}_{K_{g}} \geqslant g^{}_{\alpha} \geqslant 1^{}_{W}$ from
the proof of Proposition~\ref{prop:dens-versus-gamma}, and consider
the weighted Dirac combs
\[
    \omega^{}_{g^{}_{\alpha}} \, = \sum_{x\in {\smoplam} (K_{g})}
       g^{}_{\alpha} (x^{\star}) \, \delta^{}_{x} \ts ,
\]
where $K_{g}\subseteq H$ is the compact set introduced earlier that
covers the supports of all $g^{}_{\alpha}$.  Since each
$\omega^{}_{g^{}_{\alpha}}$, as well as $\delta^{}_{\! \vL}$, is
supported in the same Meyer set $\oplam (K_{g})$, we have pointwise
(and hence norm) convergence $\lim_{\alpha} \omega^{}_{g^{}_{\alpha}}
= \delta^{}_{\! \vL}$. Moreover, for each comb
$\omega^{}_{g^{}_{\alpha}}$, there is a hull $\XX^{}_{\alpha} =
\overline{ \{ \delta^{}_{t} * \omega^{}_{g^{}_{\alpha}} \mid t \in G
  \} }$ that is compact in the vague topology and defines a
topological dynamical system $(\XX^{}_{\alpha}, G)$. In fact, one has
more; see \cite[Thm.~3.1]{LR} as well as \cite{LS2}.

\begin{fact}\label{fact:ergodic}
  Each dynamical system\/ $(\XX^{}_{\alpha}, G)$ is minimal and admits
  precisely one\/ $G$-invariant probability measure, $\mu^{}_{\alpha}$
  say, and is thus strictly ergodic. Moreover, the system is
  topologically conjugate to its maximal equicontinuous factor,
  wherefore it has pure point diffraction and dynamical spectrum, and
  the hull possesses a natural structure as a compact Abelian
  group. \qed
\end{fact}

Clearly, the Dirac comb $\delta^{}_{\! \smoplam (K_{g})}$ is
translation bounded, so there is a compact set $K\subset G$ and a
constant $C>0$ such that $\| \delta^{}_{\! \smoplam (K_{g})} \|^{}_{K}
\leqslant C$.  By construction, we also have $\vL \subset \oplam
(K_{g})$. Consequently, both our Dirac comb $\delta^{}_{\! \vL}$ and
the measures $\omega^{}_{g^{}_{\alpha}}$ are elements of
\[
    \YY \, := \, \big\{ \nu \in \cM^{\infty} (G) \,\big|\,
    \| \nu \|^{}_{K} \leqslant C \big\}\ts ,
\]
which is a compact subset of $\cM^{\infty} (G)$. In fact, for all
$\alpha$, we have the relation
\begin{equation}\label{eq:compare}
    0 \, \leqslant \, \delta^{}_{\! \vL} \, \leqslant \,
    \omega^{}_{g^{}_{\alpha}} \, \leqslant \, \delta^{}_{\! \smoplam (K_{g})}
    \, \in \, \YY
\end{equation}
as an inequality between pure point measures. Moreover, we also have
$\XX^{}_{\vL} \subset \YY$ as well as $\XX^{}_{\alpha} \subset \YY$
for all $\alpha$.  Clearly, the measures $\mu^{}_{\alpha}$ have a
trivial extension to measures on $\YY$, still called
$\mu^{}_{\alpha}$, such that $\supp (\mu^{}_{\alpha}) =
\XX^{}_{\alpha}$.  In particular, $\omega^{}_{g^{}_{\alpha}}$ is then
generic for $\mu^{}_{\alpha}$.  We can now work within $\YY$ for
approximation purposes. To do so, we need a smoothing operation, which
is based on the linear mapping $ \phi \! : \, C_{\mathsf{c}} (G)
\xrightarrow{\quad} C (\YY) $, $c \mapsto \phi^{}_{c}$, where
\[
      \phi^{}_{c} (\nu) \, := \, \bigl( \nu * c \bigr) (0) \ts .
\]
This is the standard approach to lift continuous functions on $G$ with
compact support to continuous functions on a compact measure space
such as $\YY$. It underlies the fundamental relation between
diffraction and dynamical spectra via the Dworkin argument; compare
\cite{BL,BLvE} and references therein.

\begin{lemma}\label{lem:orbit-average-one}
  For each\/ $c\in C_{\mathsf{c}} (G)$ and each\/ $\varepsilon > 0$,
  there exists some bound\/ $M$ and some integer\/ $N$ such that
\[
     \frac{1}{\vol (A_{n})}  \int_{\nts A_{n}} \bigl|
     \phi^{}_{c} (\delta^{}_{-t} * \omega^{}_{g^{}_{\alpha}}) -
     \phi^{}_{c} (\delta^{}_{-t} * \delta^{}_{\! \vL})
     \bigr| \dd t  \, < \, \varepsilon
\]    
   holds for all\/ $\alpha \succ M$ and all\/ $n > N$.
\end{lemma}

\begin{proof}
  In view of the linearity of the mapping $c \mapsto \phi^{}_{c}$, it
  suffices to prove the claim for non-negative $c$ that are not
  identically zero, where $\| c \|^{}_{1} = \int_{G} \, c(t) \dd t >
  0$. The extension to general $c$ is then a standard $4\ts
  \varepsilon$-argument via splitting $c$ into its real and imaginary
  parts and writing a real-valued function as a difference of
  non-negative functions.

  From the proof of Proposition~\ref{prop:dens-versus-gamma}, in
  conjunction with Eq.~\eqref{eq:compare}, we know that, given
  $\varepsilon > 0$, there is an index $M$ and an integer $N$ such that
\begin{equation}\label{eq:esti-one}
    0 \, \leqslant \, \frac{ \bigl(
    \omega^{}_{g^{}_{\alpha}} - \delta^{}_{\! \vL} \bigr) (A_{n})}
    {\vol (A_{n})} \, < \,
    \frac{\varepsilon}{\ts 2 \, \| c \|^{}_{1}}
\end{equation}
holds for all $\alpha \succ M$ and $n>N$.

With the abbreviation $\nu^{}_{\alpha} := \omega^{}_{g^{}_{\alpha}}\!
- \delta^{}_{\! \vL}$, one has $0 \leqslant \phi^{}_{c}
(\delta^{}_{-t} * \omega^{}_{g^{}_{\alpha}}) - \phi^{}_{c}
(\delta^{}_{-t} * \delta^{}_{\! \vL}) = \bigl( \nu^{}_{\alpha} \nts *
c \bigr) (t)$. If $B = \supp (c)$, it is clear that the two functions
$(\nu^{}_{\alpha}\nts * c)|^{}_{A_{n}}$ and $\nu^{}_{\alpha,n}\nts *
c$ agree on the complement of the compact set $\partial^{B}\! A_{n}$,
where $\nu^{}_{\alpha,n} := \nu^{}_{\alpha}|^{}_{A_{n}}$.  Thus, for
each $\alpha$,
\[
      \lim_{n\to\infty}  \frac{1}{\vol (A_{n})}
      \int_{G} \bigl( (\nu^{}_{\alpha,n}\nts * c) (t)   -
      (\nu^{}_{\alpha}\nts * c)|^{}_{A_{n}} (t) \bigr)
      \dd t \, = \, 0
\]
as a consequence of the van Hove property of $\cA$. The limit is even
uniform in $\alpha$, because the absolute value of the integral, which
effectively only runs over the set $\partial^{B}\! A_{n}$, is bounded
by $2 \ts \| \delta^{}_{\! \smoplam (K_{g})} \|^{}_{B} \, \| c
\|^{}_{\infty} \, \vol ( \partial^{B}\! A_{n} )$ as a result of
Eq.~\eqref{eq:compare} in conjunction with Fact~\ref{fact:bounds}, the
latter applied with $\mu = \delta^{}_{0}$.  So, possibly after
adjusting $N\!$, we know that
\begin{equation}\label{eq:esti-two}
    \biggl|\frac{1}{\vol (A_{n})}
      \int_{G} \bigl( (\nu^{}_{\alpha,n}\nts * c) (t)   -
      (\nu^{}_{\alpha}\nts * c)|^{}_{A_{n}} (t) \bigr)
      \dd t\biggr| \, < \, 
     \frac{\varepsilon}{2}
\end{equation}
holds for all $n > N$ and all $\alpha$.

Now, since also $ \bigl(\nu^{}_{\alpha,n}\nts * c \bigr) (t) \geqslant
0$, one finds
\[
\begin{split}
     0 \, &  \leqslant \, \frac{1}{\vol (A_{n})} \int_{\nts A_{n}}
       \bigl( \phi^{}_{c} (\delta^{}_{-t} * \omega^{}_{g^{}_{\alpha}}) -
       \phi^{}_{c} (\delta^{}_{-t} * \delta^{}_{\! \vL}) \bigr)
       \dd t \\[1mm]
       & \leqslant \, \frac{1}{\vol (A_{n})} 
       \int_{\nts A_{n}} \bigl(\nu^{}_{\alpha,n}\nts * c \bigr) (t)
        \dd t \;  + \; \frac{\varepsilon}{2} \\[1mm]
      & =  \, \frac{1}{\vol (A_{n})} 
      \int_{G} \int_{G} 1^{}_{A_{n}} (t) \, c (t-y) \dd t
      \dd \nu^{}_{\alpha,n} (y) \;  + \; \frac{\varepsilon}{2} \\[1mm]
  & =  \, \frac{1}{\vol (A_{n})} \int_{G}
     \nu^{}_{\alpha,n} (A_{n} - t') \, c (t') \dd t'
       \;  + \; \frac{\varepsilon}{2} \\[1mm]
  & \leqslant \, \frac{ \nu^{}_{\alpha} (A_{n})}
    {\vol (A_{n})} \, \| c \|^{}_{1} 
    \; + \; \frac{\varepsilon}{2}
     \; < \; \varepsilon
\end{split}
\]       
where the second line follows from the inequality in
Eq.~\eqref{eq:esti-two}.  Fubini's theorem gives the ensuing identity,
while the next step results from setting $t' = t-y$ and applying
Fubini again. Since
\[
    0 \, \leqslant \,  \nu^{}_{\alpha,n} (A_{n} - t')
    \, \leqslant \, \nu^{}_{\alpha,n} (G)
    \, = \, \nu^{}_{\alpha} (A_{n}) \ts ,
\]
the last estimate is a consequence of Eq.~\eqref{eq:esti-one}.
\end{proof}

\begin{coro}\label{coro:orbit-average-one}
  For each\/ $c\in C_{\mathsf{c}} (G)$ and each\/ $\varepsilon > 0$,
  there exists some bound\/ $M$ and some integer\/ $N$ such that, for
  any\/ $h \in C_{\mathsf{u}} (G)$,
\[
     \frac{1}{\vol (A_{n})}  \int_{\nts A_{n}} \bigl|
     \bigl( \phi^{}_{c} (\delta^{}_{-t} * \omega^{}_{g^{}_{\alpha}}) -
     \phi^{}_{c} (\delta^{}_{-t} * \delta^{}_{\! \vL}) \bigr) \, h(t)
     \bigr| \dd t \, < \, \varepsilon\, \| h \|^{}_{\infty}
\]    
   holds for all\/ $\alpha \succ M$ and all\/ $n > N$.
\end{coro}

\begin{proof}
One clearly has the estimate
\[
\begin{split}
     \frac{1}{\vol (A_{n})}  \int_{\nts A_{n}} &
     \bigl| \bigl( \phi^{}_{c} (\delta^{}_{-t} * \omega^{}_{g^{}_{\alpha}}) -
     \phi^{}_{c} (\delta^{}_{-t} * \delta^{}_{\! \vL}) \bigr) \, h(t)
     \bigr| \dd t  \\[2mm]
     & \leqslant \, \frac{\| h \|^{}_{\infty}}{\vol (A_{n})}  
     \int_{\nts A_{n}} \bigl|
     \phi^{}_{c} (\delta^{}_{-t} * \omega^{}_{g^{}_{\alpha}}) -
     \phi^{}_{c} (\delta^{}_{-t} * \delta^{}_{\! \vL}) \bigr| 
     \dd t \ts ,
\end{split}
\]
where $\| h \|^{}_{\infty} < \infty$ by our assumptions. The claim now
follows from Lemma~\ref{lem:orbit-average-one}.
\end{proof}

The next step consists in extending the orbit average estimate to a
sufficiently large class of functions so that we can later apply the
Stone--Weierstrass theorem.  For this, we first need the algebra
generated by functions of type $\phi^{}_{c}$.  Since $\phi$ is a
linear map, we only need to extend Lemma~\ref{lem:orbit-average-one}
to products of such functions as follows.

\begin{prop}\label{prop:orbit-average}
  Let\/ $k\in\NN$ and choose arbitrary functions\/ $c^{}_{1}, \ldots ,
  c^{}_{k} \in C_{\mathsf{c}} (G)$ and\/ $\varepsilon > 0$. Then,
  there exists some bound\/ $M$ and some integer\/ $N$ such that
\[
     \frac{1}{\vol (A_{n})}  \int_{\nts A_{n}}
     \biggl| \, {\textstyle \prod\limits_{i=1}^{k}} \phi_{c^{}_{i}} 
     (\delta^{}_{-t} * \omega^{}_{g^{}_{\alpha}}) \, -
      {\textstyle \prod\limits_{j=1}^{k}} \phi_{c^{}_{j}} 
      (\delta^{}_{-t} * \delta^{}_{\! \vL})  \biggr|
     \dd t \, < \, \varepsilon
\]    
   holds for all\/ $\alpha \succ M$ and all\/ $n > N$.
\end{prop}

\begin{proof}
  For $k=1$, the claim is just Lemma~\ref{lem:orbit-average-one} and
  hence true.  Assume the claim to hold for some $k\in\NN$ and
  consider the case $k+1$.  For any $\nu \in \YY$ and $c\in
  C_{\mathsf{c}} (G)$, the function defined by $t \mapsto \phi^{}_{c}
  (\delta^{}_{-t} * \nu)$ is an element of $C_{\mathsf{u}} (G)$.  So,
  consider
\[
      0\, \leqslant \, C \, := \, 
      \max_{1 \leqslant i \leqslant k+1} \, \sup_{\nu\in\YY}\,
      \lvert \phi^{}_{c_{i}} (\nu) \rvert \, < \, \infty \ts ,
\]
which is an upper bound to the absolute value of $\phi^{}_{c_{i}}$
along any orbit in $\YY$.

Now, let $g=\prod_{i=1}^{k} \phi^{}_{c_{i}}$ and
$h=\phi^{}_{c_{k+1}}$.  Choose $M$ and $N$ such that
\[
   \frac{1}{\vol (A_{n})} \int_{\nts A_{n}} 
   \bigl| g(\delta^{}_{-t} * \omega^{}_{g^{}_{\alpha}}) -
   g(\delta^{}_{-t} * \delta^{}_{\! \vL}) \bigr|
   \dd t  \, < \,
   \frac{\varepsilon}{2\ts (C+1)}\;
\]
and
\[
   \frac{1}{\vol (A_{n})} \int_{\nts A_{n}} 
   \bigl| h(\delta^{}_{-t} * \omega^{}_{g^{}_{\alpha}}) -
   h(\delta^{}_{-t} * \delta^{}_{\! \vL})\bigr|
   \dd t   \, < \,
   \frac{\varepsilon}{2\ts (C +1)^{k}} 
\]
holds for all $\alpha \succ M$ and $n>N$, which is possible 
under our assumptions.

Then, we have
\[
\begin{split}
   \frac{1}{\vol (A_{n})} \int_{\nts A_{n}} &
   \bigl| g(\delta^{}_{-t} * \omega^{}_{g^{}_{\alpha}})\,
   h(\delta^{}_{-t} * \omega^{}_{g^{}_{\alpha}}) -
   g(\delta^{}_{-t} * \delta^{}_{\! \vL})\,
   h(\delta^{}_{-t} * \delta^{}_{\! \vL})\bigr|
   \dd t  \\[2mm]
   & \leqslant \,
   \frac{1}{\vol (A_{n})} \int_{\nts A_{n}} 
   \bigl| \bigl( g(\delta^{}_{-t} * \omega^{}_{g^{}_{\alpha}}) - 
   g(\delta^{}_{-t} * \delta^{}_{\! \vL}) \bigr) \,
   h(\delta^{}_{-t} * \delta^{}_{\! \vL}) \bigr|
   \dd t  \\[2mm]
   & \quad + \frac{1}{\vol (A_{n})} \int_{\nts A_{n}} 
   \bigl| \bigl( h(\delta^{}_{-t} * \omega^{}_{g^{}_{\alpha}}) - 
   h(\delta^{}_{-t} * \delta^{}_{\! \vL})\bigr) \,
   g(\delta^{}_{-t} * \delta^{}_{\! \vL}) \bigr|
   \dd t \, < \,  \varepsilon
\end{split}
\]
by our previous assumptions. Here, the second term is estimated by an
application of Corollary~\ref{coro:orbit-average-one}, with
$\sup_{t\in G} \lvert g (\delta^{}_{\! \vL-t} )\rvert \leqslant C^{k}$,
while the estimate of the first term works as in the proof of the same
corollary.
\end{proof}

At this point, we may consider the algebra $\AAA$ of continuous
functions on $\YY$ that is generated by the functions $\phi^{}_{c}$
with arbitrary $c\in C_{\mathsf{c}} (G)$ together with the constant
function $1$.  This algebra is dense in $C (\YY)$ by the
Stone--Weierstrass theorem.  We thus have a suitable algebra of
functions at our disposal to assess equality of probability measures
on $\YY$. Moreover, via Proposition~\ref{prop:orbit-average}, we will
be able to assess ergodicity properties as well. In our present
context, it would suffice to consider real-valued functions, as all
our measures will be positive or signed. Nevertheless, we will discuss
the general case of complex-valued functions, as this causes no extra
complications.

To continue, observe that we have a net $(\mu^{}_{\alpha})$ of
invariant probability measures on $\YY$, which is compact.  There is a
converging subnet which defines a measure $\mu$ that is also
$G$-invariant. Our next step will be to show that this measure is
ergodic and that the net itself converges to $\mu$, so $\mu$ is
unique.

\begin{prop}\label{prop:Cauchy}
  Consider the net\/ $(\mu^{}_{\alpha})$ of ergodic\/ $G$-invariant
  probability measures on\/ $\YY$. Then, for all\/ $c^{}_{1}, \ldots ,
  c^{}_{k} \in C_{\mathsf{c}} (G)$, the net\/ $\bigl(\mu^{}_{\alpha}
  (\phi^{}_{c^{}_{1}} \! \nts \cdot \ldots \cdot
  \phi^{}_{c^{}_{k}})\bigr)$ is a Cauchy net and hence convergent.
  Moreover, the limit satisfies
\[
     \lim_{\alpha} \mu^{}_{\alpha} \biggl(\,
     {\textstyle \prod\limits_{i=1}^{k}} \phi^{}_{c_{i}} \biggr)
     \; = \, \lim_{n\to\infty} \frac{1}{\vol (A_{n})}
     \int_{\nts A_{n}}\, {\textstyle \prod\limits_{i=1}^{k}}
     \phi^{}_{c_{i}} 
     (\delta^{}_{\! \vL - t})\, \dd t \ts .
\]   
\end{prop}

\begin{proof}
  Since $(\XX^{}_{\alpha},G,\mu^{}_{\alpha})$ is uniquely ergodic by
  Fact~\ref{fact:ergodic}, we have
\begin{equation}\label{eq:ergodic}
   \lim_{n\to\infty} \frac{1}{\vol (A_{n})}
     \int_{\nts A_{n}}\, {\textstyle \prod\limits_{i=1}^{k}}
     \phi^{}_{c_{i}} 
     (\delta^{}_{-t} * \omega^{}_{g^{}_{\alpha}})\, 
     \dd t \, = \, \mu^{}_{\alpha} \biggl(\,
     {\textstyle \prod\limits_{i=1}^{k}} \phi^{}_{c_{i}} \biggr)
\end{equation}
by the stronger version of Birkhoff's ergodic theorem for the orbit
average of a continuous function. This holds for any $\alpha$.

For a suitable index $M$ and arbitrary $\alpha, \alpha' \succ M$, we
can now estimate the difference $\big\lvert \mu^{}_{\alpha} \bigl(
\prod_{i=1}^{k} \, \phi^{}_{c_{i}}\bigr) - \mu^{}_{\alpha'} \bigl(
\prod_{i=1}^{k} \, \phi^{}_{c_{i}}\bigr) \big\rvert$ by means of a
$4\ts \varepsilon$-argument on the basis of
Proposition~\ref{prop:orbit-average} and Eq.~\eqref{eq:ergodic}. This
establishes the Cauchy property by standard arguments.

The second claim is another $3\ts \varepsilon$-argument of a similar
kind, again using Eq.~\eqref{eq:ergodic} and
Proposition~\ref{prop:orbit-average}. We leave the details to the
reader.
\end{proof}

\begin{theorem}\label{thm:generic}
  The net\/ $(\mu^{}_{\alpha})$ of ergodic, $G$-invariant probability
  measures from Proposition~$\ref{prop:Cauchy}$ converges, and the
  limit, $\mu$ say, is an ergodic, $G$-invariant probability measure
  on\/ $\YY$. Moreover, our weak model set\/ $\vL = \oplam (W)$ of
  maximal density is generic for\/ $\mu$.
\end{theorem}

\begin{proof}
  Let $\mu$ be the limit of a fixed subnet of $(\mu^{}_{\alpha})$.
  Since all $\mu^{}_{\alpha}$ are $G$-invariant probability measures
  on $\YY$ and this property is preserved under vague limits, $\mu$ is
  a $G$-invariant probability measure as well. Via
  Proposition~\ref{prop:Cauchy}, we know the evaluation of $\mu$ on
  all elements of $\AAA$, which is dense in $C (\YY)$ and thus
  determines $\mu$ completely.

  As a consequence, the limit $\mu'$ of any other convergent subnet of
  $(\mu^{}_{\alpha})$ must agree with $\mu$ on $\AAA$, whence $\mu =
  \mu'$, and our original net $(\mu^{}_{\alpha})$ is convergent, with
  limit $\mu$. Our construction thus determines a unique measure $\mu$
  on $\YY$. As a vague limit of ergodic measures, it is ergodic as
  well.

For all $f\in\AAA$, we also know from Proposition~\ref{prop:Cauchy}
that
\begin{equation}\label{eq:patch}
    \mu (f) \, = \lim_{n\to\infty}
    \frac{1}{\vol (A_{n})} 
    \int_{\nts A_{n}} f(\delta^{}_{\! \vL - t})
    \dd t \ts ,
\end{equation}
whence this also holds for all $f\in C(\YY)$. Consequently, 
$\vL$ is generic for $\mu$, which completes our argument.
\end{proof}

\begin{remark}
  Since the measure $\mu$ constructed above is a regular Borel
  measure, we can use Eq.~\eqref{eq:patch} to give a nice geometric
  interpretation for $\mu$. As our weak model set $\vL$ of maximal
  density is a point set of finite local complexity, $\mu$ induces a
  unique probability measure $\mu^{}_{0}$ the discrete hull $\XX_{0}
  := \{ \vL' \in \XX_{\vL} \mid 0 \in \vL' \}$ via a standard
  filtration. Now, $\mu^{}_{0}$ is specified by its values on the
  cylinder sets $Z_{K,\cP} = \{ \vL' \in \XX_{0} \mid \vL' \cap K =
  \cP \}$, where $K\subseteq G$ is compact and $\cP$ a $K$-cluster of
  $\vL$. An inspection of Eq.~\eqref{eq:patch} reveals that the
  measure of $Z_{K,\cP}$ is nothing but the cluster frequency of
  $\cP$, defined with respect to the van Hove sequence $\cA$. So, our
  measure $\mu$ is the \emph{cluster} (or patch) \emph{frequency
    measure} for $\vL$ relative to $\cA$.  \exend
\end{remark}

Our approach started with an individual weak model set $\vL$ of
maximal density, which is then pure point diffractive by
Theorem~\ref{thm:weak-pp}.  Now, we also have a measure-theoretic
dynamical system $(\XX^{}_{\vL}, G, \mu)$ with an ergodic measure
$\mu$ as constructed above.  Relative to the van Hove sequence $\cA$,
it is the cluster frequency measure.

Moreover, our weak model set of maximal density is generic for this
measure $\mu$ by Theorem~\ref{thm:generic}, so we know that the
individual autocorrelation $\gamma^{}_{\! \vL}$ of $\vL$ is also the
autocorrelation of the dynamical system, and its Fourier transform,
$\widehat{\ts \gamma^{}_{\! \vL}\ts }$, is the diffraction measure
both of $\vL$ and of our dynamical system \cite{BL}. Note that the
equivalence theorem only needs genericity, but not ergodicity, though
the possible statements on the diffraction of a given element of the
hull is then even weaker.

By the general equivalence theorem between diffraction and dynamical
spectrum in the pure point situation \cite{BL,BLvE}, we thus have the
following consequence.

\begin{coro}
  Let\/ $\vL$ be a weak model set of maximal density, relative to a
  fixed van Hove averaging sequence\/ $\cA$, for a CPS\/ $(G,H,\cL)$
  as above. Then, $\vL$ is pure point diffractive and the dynamical
  system\/ $(\XX^{}_{\vL}, G, \mu)$ with the measure\/ $\mu$ from
  Theorem~$\ref{thm:generic}$ has pure point dynamical spectrum.  \qed
\end{coro}

\begin{remark}\label{rem:spec}
  If the CPS is irredundant in the sense of \cite{Martin}, see also
  \cite{BLM}, we can also immediately give the dynamical spectrum,
  which (in additive notation) is $L^{0}=\pi (\cL^{0})$, with the
  lattice $\cL^{0}$ from the dual CPS in Eq.~\eqref{eq:dual-candp}.
  \exend
\end{remark}

As in Section~\ref{sec:diffraction}, general statements seem difficult
when the maximality condition for $\dens (\vL)$ is violated. Examples
for the possible complications can once again be taken from the family
of Toeplitz sequences, viewed as (weak) model sets with proper windows
\cite{BLO}. Still, repeating our above analysis for weak model sets of
minimal density, one obtains the following analogous result.

\begin{coro}\label{coro:dyn-spec}
  Let\/ $\vL$ be a weak model set of minimal density, relative to a
  fixed van Hove averaging sequence\/ $\cA$, for a CPS\/ $(G,H,\cL)$
  as above. Then, the autocorrelation of\/ $\vL$ relative to\/ $\cA$
  exists, and $\vL$ is pure point diffractive. Moreover, the dynamical
  system\/ $(\XX^{}_{\vL}, G, \mu)$, where\/ $\mu$ is the cluster
  frequency measure relative to\/ $\cA$, has pure point dynamical
  spectrum. The spectrum can be calculated as in
  Remark~$\ref{rem:spec}$. \qed
\end{coro}

\begin{remark}
  It is important to note that Corollary~\ref{coro:dyn-spec} is a
  result on the measure-theoretical eigenvalues. It may indeed happen
  (as in the visible lattice points of Section~\ref{sec:visible} and
  their arithmetic generalisations) that the topological point
  spectrum is trivial. A difference between topological and
  measure-theoretic spectrum is also well-known and studied in the
  theory of Toeplitz sequences (see \cite{DFM} and references
  therein), which can be described as weak model sets as well
  \cite{BLO}.  \exend
\end{remark}

Let us turn our attention to a versatile class of point sets that
comprise the arithmetic example of Section~\ref{sec:visible} as well
as its various generalisations.

\section{Application to coprime sublattice 
families}\label{sec:families}

Given a lattice $\vG \subset \RR^{d}$, we consider a countable
family of proper sublattices $(\vG_{n})_{n\in\NN}$ with the
coprimality property that
\begin{equation}\label{eq:simple-coprime}
    \vG_{i} + \vG_{j} \, = \, \vG
\end{equation}    
holds for all $i\ne j$. In fact, with $\vG^{}_{\! F} := \bigcap_{n\in
  F} \vG_{n}$ for $F\subset \NN$ finite and $\vG^{}_{\! \varnothing}
:= \vG$, we further assume the validity of
\begin{equation}\label{eq:lat-sections}
     \vG^{}_{\! F} + \vG^{}_{\! F'} \, = \, \vG^{}_{\! F\cap F'}
\end{equation}
for all finite $F,F'\subset \NN$, which represents some general
gcd-law of our lattice family. Finally, we assume the (absolute)
convergence condition
\begin{equation}\label{eq:abs-conv}
     \sum_{n\in\NN} \frac{1}{[\vG : \vG_{n}]} \, < \, \infty \ts .
\end{equation}
We call such a system a \emph{coprime sublattice family}, which, by
definition, is infinite.  Similar to the case of the visible lattice
points, this setting gives rise to the set
\begin{equation}\label{eq:def-coprime}
    V \, = \, \vG\setminus\textstyle{\bigcup_{n\in \NN}} \, \vG_n .
\end{equation}
Let us note in passing that a \emph{finite} family would simply result
in a crystallographic (or fully periodic) point set, hence not to a
situation outside the class of regular model sets. This is not of
interest to us here.

The coprimality condition clearly implies the Chinese remainder
theorem for pairwise coprime sublattices, hence
\[
    \vG / \vG^{}_{\! F} \, = \,
    \vG \big/ {\textstyle \bigcap_{n\in F}} \, \vG_n
    \,\simeq\ts \prod_{n\in F} \vG/\vG_n
\]
for any finite subset $F\subset \NN$. In particular, one has the index
formula
\begin{equation}\label{eq:index}
   \bigl[ \vG : \vG^{}_{\! F} \bigr]
    \, = \ts \prod_{n\in F} \bigl[ \vG:\vG_n \bigr] \ts .
\end{equation}
In particular, the lattices $\vG_{n}$ are mutually commensurate, and
any finite subset of them still has a common sublattice of finite
index in $\vG$.  Observing the relation $\bigl(\bigcap_{n\in F}
\vG^{}_{n}\bigr)^{*} = \bigplus_{n\in F} \vG^{*}_{n}$ for the dual
lattices in this situation, where $\vG^{*}_{n} \cap \vG^{*}_{m} =
\vG^{*}$ for $m\neq n$ as a result of the coprimality condition
\eqref{eq:simple-coprime}, Eq.~\eqref{eq:index} is equivalent to
\[
  \bigl[ \vG^{*}_{\! F} : \vG^{*} \bigr] \, = \,
  \big[\bigl({\textstyle \bigplus_{n\in F}}\, \vG^{*}_{n}
  \bigr) : \vG^{*} \big]
   \,=\ts \prod_{n\in F} \bigl[\vG^{*}_{n} : \vG^{*} \bigr] \ts .
\]

What is more, $V$ gives rise to a CPS $(G,H,\cL)$ as in
Eq.~\eqref{eq:candp}, with $G=\vG$ and the compact group
$H:=\prod_{n\in \NN} \vG\!  /\nts \vG_n$, where $\vG\!  /\nts \vG_n$
is a quotient group of order $[\vG:\vG_n] = \lvert \det (\vG_n) \rvert
/ \lvert \det ( \vG)\rvert$.  The lattice is given by
$\cL=\{(x,\iota(x))\mid x\in\vG\}$, where the $\star$-map $\iota$ is
again the natural (diagonal) embedding of $\vG$ in $H$ ($x\mapsto
(x+\vG_n)_{n\in\NN}$).  Indeed, we can write $V$ as a cut and project
set, $V = \oplam (W)$, with window
\begin{equation}\label{eq:lat-win}
   W \, = \prod_{n\in \NN} \bigl( (\vG\!  
   /\nts \vG_n) \setminus\{0+\vG_n\} \bigr) .
\end{equation}
With standard arguments, and using the convergence condition
\eqref{eq:abs-conv}, it is easy to verify the following result.

\begin{fact}\label{fact:win-vol}
  The window\/ $W$ of Eq.~\eqref{eq:lat-win} is a compact subset of\/
  $H$ with empty interior, and has positive measure
\[
    \theta^{}_{\! H} (W) \, = \prod_{n\in \NN}
   \Bigl(1-\frac{1}{[\vG:\vG_n]}\Bigr) 
\]
  with respect to the normalised Haar measure\/ $\theta^{}_{\! H}$ 
  on\/ $H$.   \qed
\end{fact}

The following result characterises the cases where $V$ is of maximal
density. Throughout, we let
 $\cA = (A_{m})^{}_{m\in\NN}$ be a van Hove sequence of centred
   balls, with\/ $A_{m} = B_{m} (0)$, say.

\begin{prop}
  The weak model set\/ $V= \oplam (W)$ is of maximal natural density
  for the CPS\/ $(\vG,H,\cL)$ constructed above if and only if
 \begin{equation}\label{eq:maxdens}
      \lim_{N\to\infty} \overline{\dens} \left(\left(
     \textstyle{\bigcup_{n > N}}\,  \vG_n \right) \setminus 
     \textstyle{\bigcup_{n \leqslant N}}\,  \vG_n \right)
     \, = \, 0 \ts .
  \end{equation}
\end{prop}

\begin{proof}
  For all $N\ge 1$, one has $\left( \textstyle{\bigcup_{n > N}}\,
    \vG_n \right) \setminus \textstyle{\bigcup_{n \leqslant N}}\,
  \vG_n=\left( \textstyle{\bigcup_{n > N}}\, \vG_n \right) \cap
  \left(\vG \setminus \textstyle{\bigcup_{n \leqslant N}}\,
    \vG_n\right) $ and $V=\vG\setminus\textstyle{\bigcup_{n\in \NN}}
  \, \vG_n=V^{}_{\! N}\setminus R^{}_{\! N}$, where
\[
    V^{}_{\! N}\, = \vG \setminus \textstyle{\bigcup_{n
        \leqslant  N}}\,  \vG_n \quad \text{and} \quad R^{}_{\! N}\, =
    \left(\textstyle{\bigcup_{n > N}}\,  \vG_n\right) \cap V^{}_{\!
      N}\ts .
\]
For fixed $N\in\NN$, it is now clear that
\[
   \underline\dens(V^{}_{\! N})-\overline\dens(R^{}_{\! N}) \, \leqslant
   \, \underline\dens(V) \, \leqslant \, \overline\dens(V) \, \leqslant
   \, \overline\dens(V^{}_{\! N})-\underline\dens(R^{}_{\! N})\ts .
\]
The set $V^{}_{\! N}$ is crystallographic with lattice of periods
$\bigcap_{n\leqslant N} \vG_{n}$. Consequently, the natural density of
$V^{}_{\! N}$ exists, so $\underline\dens(V^{}_{\!
  N})=\overline\dens(V^{}_{\! N})=\dens(V^{}_{\! N})$.  By the
inclusion-exclusion formula for sublattice densities, $\dens(V^{}_{\!
  N})$ can be computed as
\begin{equation}\label{eq:conv}
\dens(V^{}_{\! N}) \, = \,
   \frac{1}{\lvert\det (\vG)\rvert }\prod_{n\le N}
   \bigl(1-\tfrac{1}{[\vG:\vG_n]}\bigr)
   \, \searrow \, \frac{1}{\lvert\det (\vG) \rvert}\prod_{n\in \NN}
   \bigl(1-\tfrac{1}{[\vG:\vG_n]}\bigr)\ts .
\end{equation}
Thus, by the convergence from Eq.~\eqref{eq:conv}, the density of $V$
exists and is equal to
\begin{equation}\label{density}
  \dens (\cL) \,  \theta^{}_{\! H} (W) \,=\,
  \frac{1}{\lvert \det ( \vG) \rvert }\prod_{n\in \NN}
   \Big(1-\frac{1}{[\vG:\vG_n]}\Big)
\end{equation}
if and only if $\lim_{N\to\infty} \overline\dens(R^{}_{\! N})=0$. In
fact, since $V^{}_{\! N}=V \ts \dot{\cup}\ts R^{}_{\! N}$, the density
of $V$ exists if and only if the density of $R^{}_{\! N}$ exists for
all $N\in\NN$.
\end{proof}

In particular, $V$ is of maximal density if the lattice family has
\emph{light tails} in the sense of~\cite{BKKL}, which means that
\[
\lim_{N\to\infty} \overline{\dens} \left(
     \textstyle{\bigcup_{n > N}}\,  \vG_n \right) 
     \, = \, 0 \ts .
\]
This is somewhat reminiscent of the situation of regular versus
irregular Toeplitz sequences when described as model sets;
see~\cite{BJL}. Note that the complement set $\vG\setminus V$ is
another weak model set for the same CPS, which has minimal density (in
our above terminology) if and only if $\ts V$ has maximal density.

Via our general spectral results from Sections~\ref{sec:diffraction}
and \ref{sec:hull} on weak model sets of maximal density, we now get
the following consequence.

\begin{coro}
  Given the maximal density property~\eqref{eq:maxdens}, the point
  set\/ $V \subset \vG$ from Eq.~\eqref{eq:def-coprime} is pure point
  diffractive, and its hull has pure point dynamical spectrum with
  respect to the natural cluster frequency measure.
  
  The dynamical spectrum is\/ $\varSigma = \bigplus_{n\in\NN}
  \vG^{*}_{n}$, and the diffraction measure of\/ $\delta^{}_{V}$ reads
\[
     \widehat{\gamma^{}_{V}} \, = \sum_{k\in\varSigma}
     \lvert a (k) \rvert^{2} \, \delta^{}_{k} \ts ,
     \quad \text{with } \; a (k) \, = \, \dens (V) 
     \prod_{n\in F^{}_{k}} \frac{1}{1 - [\vG : \vG_{n}]}  \ts ,
\]  
  where\/ $F^{}_{k} := \bigcap \{ F \subset \NN \text{ finite} \mid
  k \in \bigplus_{n\in F} \vG^{*}_{n} \}$, for each\/
  $k \in \varSigma$, is a unique finite subset of\/ $\NN$.
\end{coro}

\begin{proof}
  The first claim on the pure point nature of the two types of spectra
  is clear from our above derivation. The calculation of the dynamical
  spectrum is an elementary consequence of $\bigplus_{n\leqslant N}
  \vG^{*}_{n}$ being the spectrum of the crystallographic set
  $V^{}_{\! N}$ and taking the limit.  Also the diffraction measure
  can be calculated this way. Clearly, any point $k\in\varSigma$ is
  contained in a set of the form $\vG^{*}_{\! F} = \bigplus_{n\in F}
  \vG^{*}_{n}$ for some finite $F\subset \NN$.  Dualising
  Eq.~\eqref{eq:lat-sections} gives the relation $\vG^{*}_{\! F} \cap
  \vG^{*}_{\! F'} = \vG^{*}_{\! F\cap F'}$, which implies the claim on
  $F^{}_{k}$.
  
  Since we now know that the diffraction measure is pure point, we
  know its general form and only need to calculate the amplitude
  $a(k)$ for a given $k\in\varSigma$. This can be done by a simple
  inclusion-exclusion argument as follows, which is justified by the
  norm convergence of the sequence of approximating crystallographic
  systems obtained by suitable truncation. Observe that we have
  $\delta^{}_{V} = \sum_{F \subset \NN} \, (-1)^{\card (F)}
  \delta^{}_{\vG^{}_{\! F}}$.  This gives, in the sense of tempered
  distributions, $\widehat{\delta^{}_{V}} = \sum_{F\subset \NN}\,
  (-1)^{\card (F)} \dens (\vG^{}_{\! F}) \, \delta^{}_{\vG^{*}_{\!
      F}}$ by an application of Poisson's summation formula; compare
  \cite[Thm.~9.1]{TAO}.  From the structure of the lattices
  $\vG^{*}_{\! F}$, one now obtains the amplitude as
\[
\begin{split}
   a(k) \, & =  \sum_{F \supseteq F^{}_{k} }
   \frac{(-1)^{\card (F)} \dens (\vG)}
   {\prod_{m\in F} \, [ \vG : \vG_{m}]} \, = \,
   \frac{(-1)^{\card (F^{}_{k})} \dens (\vG)} 
   {\prod_{n \in F^{}_{k} } \, [\vG : \vG_{n}] }
   \sum_{F \subseteq \NN \setminus F^{}_{k}}
   \frac{(-1)^{\card (F)}} {\prod_{m \in F}\, [\vG : \vG_{m}]} \\[2mm]
   & = \, \frac{(-1)^{\card (F^{}_{k})} \dens (\vG)} 
   {\prod_{n \in F^{}_{k} } \, [\vG : \vG_{n}] }
   \, \prod_{m \in \NN \setminus F^{}_{k}}
   \Bigl( 1 - \frac{1}{[\vG : \vG_{m}]} \Bigr) \\[2mm]
   & = \, \biggl(\dens (\vG) \prod_{m\in \NN}
   \Bigl( 1 - \frac{1}{[\vG : \vG_{m}]} \Bigr)\biggr)
   \prod_{n\in F^{}_{k}} \frac{-1}{[\vG : \vG_{n}]}
   \Bigl( 1 - \frac{1}{[\vG : \vG_{n}]} \Bigr)^{-1} \\[2mm]
   & = \, \dens (V) \prod_{n\in F^{}_{k}}
   \frac{1}{1 - [\vG : \vG_{n}]} \ts ,
\end{split}
\]
where the sums are over finite subsets of $\NN$ and our previous
formula for the density of $V$ was used in the last step.
\end{proof}

The various examples of $\cB$-free systems and their generalisations,
which are all covered, can be seen as coprime lattice families with an
arithmetic structure. Also, they all fall into the class of weak model
sets of maximal density, and are thus special cases of the general
theory of weak model sets. In particular, their spectral properties
get a nice and general explanation in this way.  \bigskip

\section*{Acknowledgements}

It is a pleasure to thank Gerhard Keller and Christoph Richard for
helpful discussions, and for sharing their approach from reference
\cite{KR} with us prior to publication. NS was supported by NSERC,
under grant 2014-03762.  This work was also supported by the German
Research Foundation (DFG), within the CRC 701.

\end{document}